\documentclass[a4paper,9pt,draft,reqno]{amsart}

\usepackage{color}
\usepackage{latexsym}
\usepackage{amsmath}
\usepackage{amsfonts}
\usepackage{amssymb}
\usepackage{mathrsfs}

\usepackage[T1]{fontenc}

\usepackage{enumerate}

\usepackage[pdftex,pdfauthor={Denis Bonheure, Jonathan Di Cosmo, Carlo Mercuri},pdftitle={Ground NLSP}]{hyperref}
\usepackage{amsmath}
\newtheorem{Thm}{Theorem}
\newtheorem{Prop}{Proposition}[section]
\newtheorem{Lem}[Prop]{Lemma}

\theoremstyle{definition}

\theoremstyle{remark}
\newtheorem{Rem}{Remark}
	
\newcommand{\R}{\mathbb{R}}
\newcommand{\N}{\mathbb{N}}

\newcommand{\abs}[1]{\lvert #1 \rvert}

\newcommand{\norm}[1]{\left\| #1 \right\|}

\newcommand{\DR}{{\mathcal D}}
\newcommand{\Dp}[2]{\frac{\partial #1}{\partial #2}}
\newcommand{\dist}{\mathop{\mathrm{dist}}}  		

\title[Concentration on circles for nonlinear Schr\"odinger-Poisson]{Concentration on circles for nonlinear
Schr\"odinger-Poisson systems with unbounded potentials vanishing at infinity}

\author{Denis Bonheure}
\address{
  D{\'e}partement de Math{\'e}matique\\
  Universit{\'e} libre de Bruxelles, CP 214\\
  Boulevard du Triomphe, B-1050 Bruxelles, Belgium}
\email{denis.bonheure@ulb.ac.be}

\author{Jonathan Di Cosmo}
\address{
  D{\'e}partement de Math{\'e}matique\\
  Universit{\'e} catholique de Louvain\\
  Chemin du Cyclotron 2, 1348 Louvain-la-Neuve, Belgium}
\address{
  D{\'e}partement de Math{\'e}matique\\
  Universit{\'e} libre de Bruxelles, CP 214\\
  Boulevard du Triomphe, 1050 Bruxelles, Belgium}
\email{Jonathan.DiCosmo@uclouvain.be}
\thanks{Jonathan Di Cosmo is a research fellow of the Fonds de la Recherche Scientifique--FNRS}

\author{Carlo Mercuri}
\address{S.I.S.S.A./I.S.A.S.\\ Via Bonomea 265, 34136 Trieste, Italy
}
\email{mercuri@sissa.it }


\thanks{}


\date{\today}

\begin{document}
\begin{abstract}
The present paper is devoted to weighted Nonlinear Schr\"odinger- Poisson systems with potentials possibly unbounded and
vanishing at infinity. Using a purely variational approach, we prove the existence of solutions concentrating on a
circle.
\end{abstract}

\maketitle


\section{Introduction}

The aim of the present paper is to study of the behavior of a certain class of solutions for the following nonlinear Schr\"odinger-Poisson system
\begin{equation} \label{NLSP}
\left\{
\begin{array}{l}
-\varepsilon^2 \Delta u+V(x)u+\rho(x)\phi u = K(x)u^p, \,\,\,\,\,\, x \in {\mathbb R}^3,  \\
  \\
-\Delta \phi=\rho(x)u^{2},
\end{array}
\right.
\end{equation}
in the semiclassical limit, namely for $\varepsilon\rightarrow 0,$ where $\varepsilon$ stands for the reduced Planck constant $\hbar$. In particular, we focus on solutions concentrating on a circle.

Let us
choose any $1$-dimensional linear subspace $d \subset \R^3$. We denote by $\pi$
the orthogonal complement of $d$.
If $x \in \R^3$, we will write $x = (x',x'')$ with $x' \in d$ and
$x'' \in \pi$.

As a particular case of our main result, we have the following theorem:
\begin{Thm}\label{particularcase}
Let $p>3$ and $V \in C(\R^3 \backslash \left\{ 0 \right\},\R^+)$ be a
radial potential. Write $V(x) = \tilde{V}(x',\abs{x''})$. If there exists $r^* >0$ such that the function
\begin{align*}
\mathcal{M}(r) := r \left[\tilde{V}(0,r)\right]^{\frac{2}{p-1}}
\end{align*}
has an isolated local minimum at $r = r^*$ such that $\mathcal{M}(r^*)>0$, then for
$\varepsilon$ small enough, the system
\begin{equation*}
\left\{
\begin{array}{l}
-\varepsilon^2 \Delta u+V(x)u+\phi u = u^p, \,\,\,\,\,\, x \in {\mathbb R}^{3},  \\
-\Delta \phi = u^{2},
\end{array}
\right.
\end{equation*}
has a positive cylindrically symmetric solution
$u_{\varepsilon}$ that concentrates on the circle of radius $r^*$ centered at the origin and contained in the plane
$\pi$.
\end{Thm}

We point out that we have no assumption about the decay of $V$ at infinity. In particular, $V$ could be compactly supported. This is an improvement on previous works, see e.g. \cite{Ambrosetti}.

A fundamental physical problem arises from the correspondence principle, according to which quantum mechanics contains classical mechanics as $\hbar\rightarrow 0.$ In the framework of the Schr\"odinger equation with Coulomb potential one can construct solutions which are localized around classical Keplerian elliptic orbits by superposition of states of minimal quantum fluctuation (coherent states), see \cite{GayDelandeBommier, NAU}. Due to the dispersive nature of the Schr\"odinger equation, a rigorous reduction to classical mechanics cannot in general be performed. By introducing a local nonlinear homogeneous term $u^p,$ in \cite{BenciDaprile}, the authors prove the existence of solutions for the $2D$ nonlinear Schr\"odinger equation with radial potential, concentrating on a circle. In the case of radial potentials, due to the invariance by rotations, the classical and quantum angular momentum are conserved as indicated by Noether's Theorem. This suggests that the solutions concentrating on Keplerian orbits are suitable candidates in order to mimic, in the semi-classical limit, the classical dynamics described by Newton's equations. In \cite{Daprile}, the existence of solutions concentrating on circles has been obtained for the $3D$ nonlinear Schr\"odinger equation with cylindrically symmetric potential. In both  \cite{BenciDaprile,Daprile} the underlying idea is to find solutions with nonzero angular momentum. By a different method in \cite{BVS} and \cite{BDCVS} the existence of solutions concentrating on points and, respectively, on $k-$spheres has been obtained for the nonlinear Schr\"odinger equation. In particular, in \cite{BDCVS} the existence of solutions concentrating on a circle has been obtained when radial symmetry occurs, as in Theorem \ref{particularcase}. Our aim is to extend \cite{BVS,BDCVS} to the nonlinear Schr\"odinger-Poisson system.

Now we describe our assumptions.

\subsection{The potentials}\label{hyp:pot}
We consider a nonnegative potential $V \in C(\R^3 \backslash \left\{ 0 \right\})$, a nonnegative competing function
$K \in C(\R^3 \backslash \left\{ 0 \right\})$, $K \not\equiv 0$, and a weight $\rho \in L^{3/2}_{\text{loc}}(\R^3)\cap
L^{\infty}_{\text{loc}}\left(\R^3\setminus \left\{ 0 \right\}\right)$.
We assume that for every $R \in \mathbf{O}(3)$ such that $R(d)=d$, we have $V \circ R = V$, $K
\circ R = K$ and $\rho \circ R = \rho$. This will be the case if for example $V$, $K$ and $\rho$ are radial functions.

\subsection{The nonlinearity}
We consider, for simplicity, a homogeneous nonlinear term $u^{p}$ with $3<p<\infty$.
The condition $p>3$ will be needed in order to ensure the boundedness of Palais-Smale sequences.

\subsection{The growth conditions}
Let
\begin{align*}
W(x) := V(x) + \frac{\rho(x)}{ 1+\abs{x} }.
\end{align*}
Following \cite{BVS,MVS} we impose one of the three sets of growth conditions at infinity :
\begin{itemize}
\item[$(\mathcal{G}_{\infty}^1)$] there exists $\sigma < p-3$ such that
\[
  \limsup_{\abs{x} \to \infty} \frac{K(x)}{\abs{x}^\sigma} < \infty;
\]
\item[$(\mathcal{G}_{\infty}^2)$] there exists $\sigma \in \R$ such that
 \begin{align*}
  \liminf_{\abs{x} \rightarrow \infty} W(x) \abs{x}^{2} &> 0\ &&\text{ and } &  \limsup_{\abs{x} \to \infty} \frac{K(x)}{\abs{x}^\sigma} &<\infty;
 \end{align*}
\item[$(\mathcal{G}_{\infty}^3)$] there exist $\alpha < 2$ and $\sigma \in \R$ such that
 \begin{align*}
  \liminf_{\abs{x} \rightarrow \infty} W(x) \abs{x}^{\alpha} &> 0\ & & \text{ and }\ &\limsup_{\abs{x} \to \infty} \frac{K(x)}{\exp ( \sigma
\abs{x}^{\frac{2-\alpha}{2}} )} &< \infty.
 \end{align*}
\end{itemize}
Note that in comparison with \cite{BVS}, in $(\mathcal{G}_{\infty}^2)$ and $(\mathcal{G}_{\infty}^3)$, $V$ might vanish somewhere.
We also impose one of the three sets of growth conditions at the origin, which mirror those at infinity :
\begin{itemize}
\item[$(\mathcal{G}_{0}^1)$] there exists $\tau > -2$, such that
\[
   \limsup_{\abs{x} \to 0} \frac{K(x)}{\abs{x}^{\tau}} < \infty,
\]
\item[$(\mathcal{G}_{0}^2)$] there exists $\tau \in \R$ such that
 \begin{align*}
  \liminf_{\abs{x} \rightarrow 0} V(x) \abs{x}^{2} &> 0\ &&\text{ and }\ &\limsup_{\abs{x} \to 0} \frac{K(x)}{\abs{x}^{\tau}} < \infty;
 \end{align*}
\item[$(\mathcal{G}_{0}^3)$]
 there exist $\gamma > 2$ and $\tau \in \R$ such that
 \begin{align*}
  \liminf_{\abs{x} \rightarrow 0} V(x) \abs{x}^{\gamma} &> 0\ &&\text{ and }\ & \limsup_{\abs{x} \to 0} \frac{K(x)}{\exp ( \tau
\abs{x}^{-\frac{\gamma-2}{2}} )} &< \infty.
 \end{align*}
\end{itemize}

\subsection{The auxiliary potential}
Before we can state our last assumption, we need a few preliminaries. Let $a,b > 0$. We consider the \textit{limit
equation}
\begin{align}\label{plim}
	 -\Delta u + au = bu^p \hspace{1cm} \text{in} \ \R^{2}.
\end{align}
 The weak solutions of \eqref{plim} are critical
points of the functional $\mathcal{I}_{a,b} : H^1(\R^{2}) \rightarrow \R$ defined by
\begin{align}\label{defIab}
 \mathcal{I}_{a,b}(u) := \frac{1}{2} \int_{\R^{2}} \left(  \abs{\nabla u}^2 + a u^2 \right)\: dx - \frac{b}{p+1}
\int_{\R^{2}} u^{p+1}\: dx.
\end{align}
Any nontrivial critical point  $u \in H^1(\R^{2})$ of $\mathcal{I}_{a,b}$, belongs to
the Nehari manifold
\begin{align*}
 \mathcal{N}_{a,b} := \left\{ u \in H^1(\R^{2}) \ \vert\ u \not\equiv 0\ \text{and}\ \langle \mathcal{I}_{a,b}'(u),u
\rangle = 0 \right\}.
\end{align*}
A solution $u \in H^1(\R^{2})$ is a \textit{least-energy solution} of \eqref{plim} if
\begin{align*}
 \mathcal{I}_{a,b}(u) = \inf_{v \in \mathcal{N}_{a,b}} \mathcal{I}_{a,b}(v).
\end{align*}
The \textit{ground-energy function} is defined by
\begin{align*}
 \mathcal{E} : \R^+\times \R^+ \rightarrow \R^+ : (a,b) \mapsto \mathcal{E}(a,b) := \inf_{u \in \mathcal{N}_{a,b}}
\mathcal{I}_{a,b}(u).
\end{align*}
It is standard to show that
\begin{align}\label{EabMP}
 \mathcal{E}(a,b) = \inf_{\gamma \in \Gamma_{a,b}} \max_{t \in [0,1]} \mathcal{I}_{a,b} (\gamma(t)),
\end{align}
where
\begin{align*}
 \Gamma_{a,b} := \left\{ \gamma \in C([0,1],H^1(\R^2)) \ \vert\ \gamma(0)=0,\ \mathcal{I}_{a,b} (\gamma(1)) < 0
\right\}.
\end{align*}
The \textit{auxiliary potential} $\mathcal{M} :\ \R^{3} \rightarrow (0,+\infty]$ is defined by
\[
 (x',x'') \mapsto \mathcal{M}(x',x'') := \left\{ \begin{array}{ll} \abs{x''} \mathcal{E}\left(V(x),K(x)\right) &
\text{if}\ K(x) > 0, \\
 +\infty & \text{if}\ K(x)=0.
 \end{array}
 \right.
\]
The following lemma states some properties of the ground-energy function, see \cite[Lemma 3]{BVS}.
\begin{Lem}
 For every
$(a,b) \in \R^+_0 \times \R^+_0$, $\mathcal{E}(a,b)$ is a critical value of $\mathcal{I}_{a,b}$ and we have
\begin{align*}
 \mathcal{E}(a,b) = \inf_{\substack{u \in H^1(\R^2) \\ u \neq 0}} \max_{t \geq 0} \; \mathcal{I}_{a,b}(tu).
\end{align*}
If $u \in \mathcal{N}_{a,b}$ and $\mathcal{E}(a,b) = \mathcal{I}_{a,b}(u)$, then $u \in C^1(\R^2)$ and up to a
translation, $u$ is a radial function such that $\nabla u(x) \cdot x < 0$ for every $x \in \R^2\setminus \{0\}$.
Moreover, the following properties hold:
\begin{itemize}
 \item[(i)] $\mathcal{E}$ is continuous in $\R^+_0 \times \R^+_0$;
 \item[(ii)] for every $b^* \in \R^+_0$, $a \to \mathcal{E}(a,b^*)$ is strictly increasing;
 \item[(iii)] for every $a^* \in \R^+_0$, $b \to \mathcal{E}(a^*,b)$ is strictly decreasing;
 \item[(iv)] for every $\lambda > 0$, $\mathcal{E}(\lambda a, \lambda b) = \lambda^{-1/2} \mathcal{E}(a,b)$;
 \item[(v)] the ground-energy function satisfies
\begin{align*}
 \mathcal{E}(a,b) = \mathcal{E}(1,1) a^{\frac{p+1}{p-1}-1} b^{-\frac{2}{p-1}}.
\end{align*}
\end{itemize}
\end{Lem}

The last property of the preceding lemma implies the following explicit form of the auxiliary potential:
\[
\label{defM2}
 \mathcal{M}(x',x'') = \mathcal{E}(1,1) \abs{x''} \left[V(x)\right]^{\frac{p+1}{p-1} - 1}
\left[K(x)\right]^{\frac{-2}{p-1}}.
\]

Due to the symmetry that we shall impose on the solution (see \eqref{def:spaceE}), the
concentration can only occur in the plane $\pi$. We assume that there exists a smooth
bounded open set $\Lambda \subset \R^{3}$ such that
\begin{equation}
\label{hypLambda1}
\Bar{\Lambda} \cap d = \emptyset, \ \Lambda \cap \pi \neq \emptyset,
\end{equation}
for every $R \in \mathbf{O}(3)$ such that $R(d)=d$,
\begin{align}
\label{hypLambda2}
  R(\Lambda)=\Lambda
\end{align}
and the following inequalities hold
\begin{align}
 0 < \inf_{\Lambda \cap \pi}\mathcal{M} < \inf_{\partial \Lambda \cap \pi}\mathcal{M}, \label{hypLambda3} \\
 \inf_{\Lambda \cap \pi}\mathcal{M} < 2 \inf_{\Lambda} \mathcal{M}. \label{hypLambda4}
\end{align}
By continuity of $\mathcal{M}$ in $\Lambda$, this last condition is not restrictive.
Similarly, we can also assume that $V > 0$ on $\overline{\Lambda}$ and that $\mathcal{M}$ is continuous on
$\overline{\Lambda}$.

Our main result is the following.
\begin{Thm}\label{Th:main}
Let $p>3$ and $V, K$ and $\rho$ be functions satisfying the assumptions in \ref{hyp:pot}. Assume that one set
$(\mathcal{G}_{\infty}^i)$ of growth
conditions at infinity and one set $(\mathcal{G}_{0}^j)$ of growth conditions at the origin hold.
Assume also that there exists an open bounded set $\Lambda \subset \R^{3}$ such that \eqref{hypLambda1},
\eqref{hypLambda2},
\eqref{hypLambda3} and \eqref{hypLambda4} hold. Then there exists $\varepsilon_0 > 0$ such that for every
$0 < \varepsilon < \varepsilon_0$, problem \eqref{NLSP} has at least one positive solution $u_{\varepsilon}$.
Moreover, for every $0 < \varepsilon < \varepsilon_0$, there exists $x_{\varepsilon} \in \Lambda \cap \pi$
such that $u_{\varepsilon}$ attains its maximum at $x_{\varepsilon}$,
\begin{equation*}
 \liminf_{\varepsilon \to 0} u_{\varepsilon}(x_{\varepsilon}) > 0, \quad
 \lim_{\varepsilon \to 0} \mathcal{M} (x_{\varepsilon}) = \inf_{\Lambda \cap \pi} \mathcal{M},
\end{equation*}
and there exist $C>0$ and $\lambda > 0$ such that
\begin{align*}
 u_{\varepsilon}(x) &\leq C \exp{\left( -\frac{\lambda}{\varepsilon} \frac{d(x,S^1_{\varepsilon})}{1+d(x,S^1_{\varepsilon})}\right) }
\left( 1+\abs{x}^2 \right)^{\frac{-1}{2}}, &  \forall x \in \R^3,
\end{align*}
where $S^1_{\varepsilon}$ is the circle centered at the origin, contained in the plane $\pi$ and of radius
$\abs{x''_{\varepsilon}}$.
\end{Thm}

In Section $2$, we deal with an auxiliary penalized problem. This by now classical penalization argument goes back to del Pino and Felmer \cite{DF96}. The method has then been adapted in \cite{BVS,MVS} in the frame of vanishing or compactly supported potentials. Section $3$ is devoted to the asymptotic analysis of the solutions of the penalized problem while in Section $4$, we show how to go back to the original problem. At last, we give some final comments in Section $5$.
Throughout the paper, we use the following notation :
\begin{itemize}
\item[-] $\DR^{1,2}(\R^3)$ is the space
$$
\{u\in L^{2^*}(\R^3):\nabla u\in L^2(\R^3;\R^3)\}$$
equipped with the norm
$$
||u||_{\DR^{1,2}(\R^3)} :=||\nabla u||_{L^2(\R^3)} ;
$$
\item[-] $L^q_Q(\R^3)$ is the Lebesgue space of measurable functions such that
$$\int_{\R^3}Q(x)|u|^q dx<\infty.$$
\end{itemize}
As usual, $u_+:=\max(u,0)$ and $u_-:=\max(-u,0)$, $B_R$ is the open ball of radius $R$ and
$c_1,c_2,...c_j,C_1,C_2,...C_k$ always denote positive real constants.


\section{Existence for the penalized problem}
Following \cite{MVS} and \cite{BDCVS}, we define the penalization potential $H : \R^3 \to \R$ by
\begin{align*}
 H(x) := \frac{\kappa}{\abs{x}^2 \Bigl( \bigl(\log \abs{x} \bigr)^2+1 \Bigr)^{\frac{1+\beta}{2}}}
\end{align*}
where $\beta > 0$ and $0 < \kappa < \frac{1}{4}$.
Notice that for all $x \in \R^3$, we have
\[
 H(x) \leq \frac{\kappa}{\abs{x}^2}.
\]
By Hardy's inequality, we deduce that the quadratic form associated to $- \Delta - H$ is positive, i.e.
\begin{align}\label{positivity}
 \int_{\R^3} \left( \abs{\nabla u}^2 - H u^2 \right) \geq \biggl(  \frac{1}{4} - \kappa \biggr)
\int_{\R^3} \frac{\abs{u(x)}^2}{\abs{x}^2}\: dx \geq 0,
\end{align}
for all $u \in \mathcal{D}^{1,2}(\R^3)$.

This inequality implies the following comparison principle.
\begin{Prop}\label{Th:comp}
Let $\Omega \subset \R^3 \setminus\{0\}$ be a smooth domain. Let $v,w \in H^1_{\textnormal{loc}}(\Omega)\cap
C(\overline{\Omega})$ be such
that $\nabla (w-v)_- \in L^2(\Omega)$, $(w-v)_-/\abs{x} \in L^2(\Omega)$ and
\begin{align}\label{comp}
 - \Delta w - H w \geq - \Delta v - H v, \hspace{1cm} \forall x \in \Omega.
\end{align}
If $\partial\Omega \neq \emptyset$, assume also that $w \geq v$ on $\partial\Omega$. Then $w \geq v$ in $\Omega$.
\end{Prop}
\begin{proof}
 It suffices to multiply the inequality \eqref{comp} by $(w-v)_-$, integrate by parts and use \eqref{positivity}.
\end{proof}


Fix $\mu \in (0,1)$. We define the penalized nonlinearity $g_{\varepsilon}: \R^3 \times \R \rightarrow \R$ by
\[
  g_{\varepsilon}(x,s) := \chi_{\Lambda}(x) K(x) s_+^p + \left( 1-\chi_{\Lambda}(x) \right) \min\left\lbrace  \left(
\varepsilon^2 H(x)+\mu V(x)\right) s_+, K(x)s_+^p \right\rbrace.
\]
Let $G_{\varepsilon}(x,s) := \int_0^s g_{\varepsilon}(x,\sigma) d\sigma$. One can check that $g_{\varepsilon}$ is a
Carath\'eodory function with the following properties :
\begin{itemize}
 \item[($g_1$)] $g_{\varepsilon}(x,s) = o(s),\ s \rightarrow 0^+$, uniformly in compact subsets of $\R^3$.
 \item[($g_2$)] there exists $p>3$ such that
 \begin{align*}
  \lim_{s \rightarrow \infty} \frac{g_{\varepsilon}(x,s)}{s^p} = 0,
 \end{align*}
 \item[($g_3$)] there exists $2 < \theta \leq p+1$ such that
 \begin{align*}
 \begin{array}{ll}
  0 < \theta G_{\varepsilon}(x,s) \leq g_{\varepsilon}(x,s)s &\forall x \in \Lambda,\ \forall s>0, \\
  0 < 2 G_{\varepsilon}(x,s) \leq g_{\varepsilon}(x,s)s \leq \left( \varepsilon^2 H(x)+\mu V(x)\right) s^2 &\forall x
\notin \Lambda,\ \forall s>0,
 \end{array}
 \end{align*}
 \item[($g_4$)] the function
 \begin{align*}
 s \mapsto \frac{g_{\varepsilon}(x,s)}{s}
 \end{align*}
 is nondecreasing for all $x \in \R^3$.
\end{itemize}

Now we use Critical Point Theory in order to find solutions to the penalized problem
\begin{equation} \label{NLSPropen}
\left\{
\begin{array}{l}
-\varepsilon^2 \Delta u+V(x)u+\rho(x)\phi u=g_{\varepsilon}(x,u), \,\,\,\,\,\, x \in {\mathbb R}^{3},  \\
  \\
-\Delta \phi=\rho(x)u^{2}.
\end{array}
\right.
\end{equation}


For any $u^{2} \rho \in L^{1}_{loc}(\mathbb{R}^{3})$ such that
$$\int_{\mathbb{R}^{3}}\int_{\mathbb{R}^{3}}\frac{u^2(x)u^2(y)\rho(x)\rho(y)}{|x-y|}\: dx\: dy<\infty,$$
the standard distributional solution
\begin{align}\label{formula:phi}
\phi_u:=\frac{1}{4\pi |x|} \star u^2 \rho
\end{align}
belongs to
$\mathcal{D}^{1,2}(\mathbb{R}^{3})$ and is a weak solution in $\mathcal{D}^{1,2}$ (e.g. \cite{RS}).
Since we consider $u \in \mathcal{D}^{1,2}$, we will assume $\rho \in L^{3/2}_{loc}$.
%
A suitable choice of the space $X$ of admissible functions is given in the work of Ruiz \cite{RS}, in the case of
$V\equiv0$ and $\rho\equiv 1.$  Inspired by \cite{RS}, we define, for measurable $V,\rho\geq 0,$
$$\|u\|^2_{X}:=\int_{\mathbb{R}^{3}}|\nabla
u|^2+V(x)u^2\: dx+\Big(\int_{\mathbb{R}^{3}}\int_{\mathbb{R}^{3}}\frac{u^2(x)u^2(y)\rho(x)\rho(y)}{|x-y|}\: dx\: dy\Big)^{1/2}$$
and
\begin{align}
X:=\{u\in  \mathcal{D}^{1,2}(\mathbb{R}^{3})\,:\, \|u\|_{X}<\infty\}.\label{def:X}
\end{align}
As pointed out in \cite{RS}, the space
$X$ is a uniformly convex Banach space, hence it is reflexive.
Precisely, we look for solutions $(u,\phi_u)\in E\times \mathcal{D}^{1,2}(\mathbb{R}^{3}).$

We also define $H_{V,\varepsilon}$ to be the closure of $\mathcal D(\R^3)$ with respect to the norm
$$\|u\|^2_{H_{V,\varepsilon}}:=\int_{\mathbb{R}^{3}}\varepsilon^2|\nabla u|^2+V(x)u^2\: dx.$$
We will focus on the closed subspace $E\subset X$ of functions which are radial in $\pi$, namely
\begin{align}\label{def:spaceE}
 E := \bigl\{ u \in X\ \vert\ \forall R \in \mathbf{O}(3) \ \text{s.t.}\ R({d}) = {d}, u \circ R = u \bigr\}.
\end{align}
Solutions of (\ref{NLSPropen}) are the critical points of the functional $$ J_\varepsilon (u):=\frac{1}{2}\int_{{\mathbb
{R}}^{3}}(\varepsilon^2 \left|\nabla u\right|^{2}+V(x)u^{2})dx+\frac{1}{4}\int_{{\mathbb
{R}}^{3}}\phi_{u}u^{2}\rho(x)dx-
\int_{{\mathbb {R}}^{3}}G_\varepsilon (x,u) dx,$$
which is $C^1(X; \mathbb R)$.

In the present section we find critical points for $J_\varepsilon$ through a minimax scheme used in \cite{BoMerc},
modeled on \cite{AmbrosettiRa}.

The main result of this section is
\begin{Thm}\label{expena}
Assume $V,K,\rho$ satisfy the assumptions of Theorem \ref{Th:main}. Then, for any $\varepsilon>0$ and
$\kappa,\mu>0$ small enough, there exists a critical point for $J_\varepsilon$ at level
\begin{align}\label{ceps}
c_\varepsilon:=\inf_{\gamma \in \Gamma}\max_{t\in[0,1]}J_\varepsilon(\gamma(t)),
\end{align}
where
\begin{align}\label{Gammaeps}
\Gamma_{\varepsilon} := \left\{\gamma \in C([0,1],E): \gamma(0)=0, J_\varepsilon(\gamma(1))<0\right\},
\end{align}
corresponding to a nontrivial solution $(u,\phi)\in E \times \mathcal D^{1,2}(\mathbb R^3)$ for (\ref{NLSPropen}).
Moreover $u$ is positive.
\end{Thm}
\begin{Rem}\label{Palaiscrit}
{ \it Due to the invariance of the Lebesgue measure by rotations, by the
symmetric criticality principle \cite{Palais}, if $u\in E$ is critical for $J_\varepsilon  |_E,$ then $u$ is also
critical for $J_\varepsilon |_X$}.
\end{Rem}

The functional $J_\varepsilon$ has the mountain pass geometry, as it is shown in the following

\begin{Lem}\label{MMMPPP}
The functional $J_\varepsilon$ satisfies the mountain pass geometry for any $p>3,$ provided $\kappa, \mu>0$ are small enough.
Furthermore, there exists a Palais-Smale (P-S) sequence at the minimax level
 $c_\varepsilon$. In particular, defining
$$S:=\{u\in E \, :\, u_-\equiv 0\}, $$   $$S_{1/n}:=\{u\in E\, : \, \inf_{y\in S}\|u-y\|_E < 1/n\}, $$
it is possible to select the P-S sequence  $(u_n)_n $ in such a way that $u_k \in S_{1/k}$ for all $k\in \mathbb N$.
 \end{Lem}

\begin{proof}
We first prove that, for any $p>3$, the origin is a local minimum for $J_\varepsilon.$ Notice that
$\|u\|^{p+1}_{L^{p+1}_K(\Lambda)}\leq C \|u\|^{p+1}_E.$ Furthermore by Lemma \ref{Lambdaout} below, taking $\kappa, \mu>0$
small enough, we have $$J_\varepsilon(u)\geq c\int_{{\mathbb {R}}^{3}}(\varepsilon^2\left|\nabla
u\right|^{2}+V(x)u^{2})dx+\frac{1}{4}\int_{{\mathbb
{R}}^{3}}\phi_{u}u^{2}\rho(x)dx-C\|u\|^{p+1}_E.$$  Since, by definition, we have $(\int_{{\mathbb
{R}}^{3}}\phi_{u}u^{2}\rho(x)dx)^{1/2}=\|u\|^2_E-\|u\|^2_{H_{V,\varepsilon}},$ we get
\begin{eqnarray*}
J_\varepsilon(u)&\geq& c\|u\|^2_{H_{V,\varepsilon}}+\frac{1}{4}\left[\|u\|^2_E-\|u\|^2_{H_{V,\varepsilon}}\right]^2-C\|u\|^{p+1}_E
\\ &=&c\|u\|^2_{H_{V,\varepsilon}}+\frac{1}{4}\|u\|^{4}_E-\frac{1}{2}\|u\|^2_{H_{V,\varepsilon}}\|u\|^2_E+
\frac{1}{4}\|u\|^4_{H_{V,\varepsilon}}-C\|u\|^{p+1}_E.
 \end{eqnarray*}
Therefore, we get
$$
J_\varepsilon(u)\geq
c\|u\|^2_{H_{V,\varepsilon}}-\frac{\alpha^2-1}{4}\|u\|^4_{H_{V,\varepsilon}}+\frac{\alpha^2-1}{4\alpha^2}\|u\|^{4}
_E-C\|u\|^{p+1}_E.$$ Let $\|u\|^2_E<\delta.$ Then we have
$$
J_\varepsilon(u) \geq \left[c-\frac{\alpha^2-1}{4}\delta^2 \right]\|u\|^2_{H_{V,\varepsilon}} + \left[\frac{\alpha^2-1}{4\alpha^2}- C \delta^{p-3}\right]\|u\|^4_{E}.
$$
This yields, for $\alpha>1$ and $\delta$ small enough,
$$J_\varepsilon(u)\geq \left[\frac{\alpha^2-1}{4\alpha^2}- C \delta^{p-3} \right]\|u\|^4_{E}.$$
Hence, the origin is a strict local minimum point for
$J_\varepsilon.$

Moreover, $J_\varepsilon$ attains negative values along curves of the form $u_t:=t u$, with $u\in E$ such that $u^{+}\not\equiv 0$ and $t>0$. Hence $J_\varepsilon$ has the mountain pass geometry.

By the general minimax principle \cite[p.41]{MWWW}, there exists a P-S  sequence $(u_n)_n$ such that, if for $\gamma_n\in \Gamma$,
$$\max_{t\in [0,1]}J_\varepsilon (\gamma_n(t))\leq c_\varepsilon+ \frac{1}{n},$$ then
\begin{equation}\label{positivecone}
\textrm{dist}(u_n, \gamma_n([0,1]))<\frac{1}{n}.
\end{equation}
Finally, since $J_\varepsilon (u)=J_\varepsilon (|u|),$ the conclusion follows from (\ref{positivecone}).
\end{proof}

\begin{Lem}\label{Lambdaout}
For any positive constants $c>0$ there exists $\kappa(c)$ such that
\end{Lem}
\begin{eqnarray*}
c\int_{\mathbb R^3}\left|\nabla u\right|^{2} dx \geq \int_{\mathbb R^3 \setminus \Lambda} H(x)u^2 dx, \quad \forall
\kappa<\kappa(c), \,\,\forall u \in \mathcal D^{1,2}(\mathbb R^3),
\end{eqnarray*}
\begin{proof}
The claim follows directly from  Hardy's inequality.
\end{proof}

We now study some properties of the P-S sequences found in Lemma \ref{MMMPPP}.

\begin{Lem}\label{negativepartSP}
Let  $(u_n)_n $ be as in Lemma \ref{MMMPPP}  {\it such that}  $u_k \in S_{1/k}$ for all $k\in \mathbb N.$ Then
\begin{eqnarray*}
&(a)& \quad \int_{\mathbb R^3 } \phi_{(u_n)_-}(x)(u_n)_-^2 \rho(x)dx\rightarrow 0.
\end{eqnarray*}
{\it Furthermore, } $(u_n)_n $ {\it is bounded in} $E,$ {\it provided } $\kappa,\mu>0$ {\it are small enough, and we
have}
$$(b) \quad \,\,\,\, \,\int_{\mathbb R^3 } \phi_{u_n}(x)(u_n)_-^2 \rho(x)dx\rightarrow 0.$$
\end{Lem}
\begin{proof}
By definition, there exists a sequence $(y_n)_n \subset S$ such that $$\|u_n-y_n\|_E\rightarrow 0.$$
Hence $(a)$ follows:
\begin{eqnarray*}
\int_{\mathbb R^3 } \phi_{(u_n)_-}(x)(u_n)_-^2 \rho(x)dx&=&\int_{\mathbb R^3}\int_{\mathbb R^3 }
\frac{(u_n)_-^2(x)\rho(x)(u_n)_-^2(z)\rho(z)}{|x-z|}dxdz\\ &\leq&\int_{\mathbb R^3}\int_{\mathbb R^3 }
\frac{(u_n-y_n)^2(x)\rho(x)(u_n-y_n)^2(z)\rho(z)}{|x-z|}dxdz  \\&\leq &\|u_n-y_n\|^4_E \rightarrow 0\end{eqnarray*}
Notice that $(b)$ follows if we prove that $(u_n)_n$ is bounded. Indeed, define, for $f,g$ measurable and nonnegative functions, the
following quantity $$D(f,g):=\int_{\R^3}\int_{\R^3} f(x)|x-y|^{-1}g(y)dxdy.$$ From \cite[p.250]{LiebLoss},
we have
\begin{equation}\label{liebloss}
|D(f,g)|^2\leq D(f,f)D(g,g).
\end{equation}

If $(u_n)_n$ is bounded in $E,$ by the inequality above with $f:=u^2_n \rho$ and $g:=(u_n)_-^2\rho$ and by $(a)$
we have
$$
\int_{\mathbb R^3 } \phi_{u_n}(x)(u_n)_-^2 \rho(x)dx \leq  C \int_{\mathbb R^3 } \phi_{(u_n)_-}(x)(u_n)_-^2 \rho(x)dx
\rightarrow 0 $$
We now prove that $(u_n)_n$ is bounded.

Define $$\Delta_n:=\int_{\mathbb R^3}\left(g_\varepsilon(x, u_n)u_n-(p+1)G_\varepsilon(x,u_n)\right) dx$$ Using Lemma \ref{Lambdaout},
we have, choosing $\kappa,\mu>0$ small enough,
\begin{eqnarray*}
\Delta_n &\geq& -(p+1)\int_{\mathbb R^3\setminus \Lambda}G_\varepsilon(x,u_n)dx\\
&\geq& -\frac{p+1}{2}\int_{\mathbb
R^3\setminus \Lambda} [\varepsilon^2 H(x)+ \mu V(x)]u_n^2 dx\\
&\geq& - \frac{p-1}{4} \|u_n\|^2_{H_{V,\varepsilon}}.
\end{eqnarray*}
Since $(u_n)_n$ is a P-S  sequence, the above estimate yields
\begin{eqnarray}
C&\geq& (p+1)J_\varepsilon(u_n)-(J'_\varepsilon(u_n),u_n)\nonumber \\  &=&
\frac{p-1}{2}\|u_n\|^2_{H_{V,\varepsilon}}+\frac{p-3}{4}\int_{\mathbb R^3} \phi_{u_n}(x)u_n^2 \rho(x)dx + \Delta_n
\nonumber \\
&\geq& \frac{p-1}{4}\|u_n\|^2_{H_{V,\varepsilon}}+\frac{p-3}{4}\int_{\mathbb R^3} \phi_{u_n}(x)u_n^2
\rho(x)dx.\label{boundedPS}
\end{eqnarray}
As a consequence, the claim follows.
\end{proof}
In the following we shall need a family of cut-off functions.
Consider a smooth function $\zeta(r) $ such that $\zeta(r)=1$ on $[2,\infty)$ and $\zeta(r)=0$ on $[0,1].$ Then define
$$\eta_R (x):=\zeta\Big(\frac{\log (1+|x|)}{R}\Big).$$
One has \begin{equation}\label{cut}
\||x|\cdot|\nabla \eta_R(x)|\|_\infty\leq \frac{C}{R}.
\end{equation}
\begin{Lem}\label{complem}
Let $(u_n)_n \subset E$ be as in Lemma \ref{MMMPPP} and $u_n\rightharpoonup u\geq 0$ in $E$. Then, for all $\delta>0,$
there exists a ball $B\subset\R^3$ such that, for all $\kappa,\mu>0$ small enough,
\end{Lem}
\begin{eqnarray*}
  &a)& \quad \quad \limsup_ {n\rightarrow \infty}\int_{{\mathbb {R}}^{3}\setminus B}\phi_{u_{n}}(x)u^2_{n}\rho
(x)dx<\delta,\\
   &b)& \quad \quad \limsup_ {n\rightarrow \infty}\int_{{\mathbb {R}}^{3}\setminus B}V(x)u_n^2dx<\delta, \\
      &c)& \quad \quad \limsup_ {n\rightarrow \infty}\int_{{\mathbb {R}}^{3}\setminus B}H(x)u_n^2dx<\delta, \\
 &d)& \quad \quad \limsup_ {n\rightarrow \infty}\int_{{\mathbb {R}}^{3}\setminus B}\phi_{u_{n}}(x)(u_{n})_-\,u\,\rho
(x)dx<\delta,\\
 &e)& \quad \quad \limsup_ {n\rightarrow \infty}\int_{{\mathbb {R}}^{3}\setminus B}\phi_{u_{n}}(x)(u_{n})_+\,u\,\rho
(x)dx<\delta.
\end{eqnarray*}

\begin{proof}
Consider the above family of cut-off functions. We claim that, uniformly in $n$,
\begin{equation} \label {gradH}
\int_{\mathbb R^3} \Big (\frac{1}{2}|\nabla u_n|^2-H(x)u_n^2\Big )\eta^2_R dx \geq O \Big(\frac{1}{R}\Big),\quad
R\rightarrow \infty,
 \end{equation}
for all $\kappa>0$ small enough.

In order to prove this, compute $$|\nabla(u_n \eta_R)|^2=\eta_R^2 |\nabla u_n|^2+ 2u_n \eta_R \nabla u_n \nabla \eta_R+
u_n^2|\nabla \eta_R|^2.$$
We have
$$\Big |\int_{{\mathbb {R}}^{3}}u_n\eta_R \nabla u_n \nabla \eta_R dx\Big |\leq C_1 \Big |\int_{{\mathbb
{R}}^{3}}\frac{u_n}{|x|}\nabla u_n |x|\nabla \eta_R dx\Big | \leq \||x|\nabla \eta_R\|_{\infty}\Big |\int_{{\mathbb
{R}}^{3}}\frac{u_n}{|x|}\nabla u_n\,dx\Big |. $$
By (\ref{cut}), Cauchy-Schwarz and Hardy inequalities, we
obtain\begin{equation}\label{har} \Big |\int_{{\mathbb {R}}^{3}}u_n\nabla u_n \nabla \eta_R dx\Big |\leq \frac{C_2}{R}\|\nabla
u_n\|^2_2\leq\frac{C_3}{R}.\end{equation} Here we take into account that, since $u_n\rightharpoonup u$ in $E,$ $\|\nabla u_n\|_2$ is
bounded. In the same way, one can easily obtain
$$\int_{\mathbb R^3}u_n^2|\nabla \eta_R|^2\leq \frac{C_4}{R^2}.$$
Hence, by Hardy's inequality and the above estimates, we have, as $R\rightarrow \infty,$
\begin{eqnarray*}
\int_{\mathbb R^3} \Big (\frac{1}{2}|\nabla u_n|^2-H(x)u_n^2\Big )\eta^2_R dx &\geq &\int_{\mathbb
R^3}\Big(\frac{1}{2}|\nabla( u_n\eta_R)|^2 - \kappa\frac{(u_n\eta_R)^2}{|x|^2}\Big)dx+ O \Big(\frac{1}{R}\Big) \\ &\geq
& O \Big(\frac{1}{R}\Big),
 \end{eqnarray*}
and the claim follows. Furthermore, simply notice that

\begin{equation} \label{cutt}
\int_{\mathbb R^3}\nabla u_n \nabla (u_n\eta^2_R) dx=\int_{\mathbb R^3}|\nabla u_n|^2\eta_R^2 dx+ O
\Big(\frac{1}{R}\Big),\quad R\rightarrow \infty,
 \end{equation}
Finally, by (\ref{gradH}) and (\ref{cutt}), we have, for $\mu,\kappa>0$ small enough,
\begin{eqnarray*}
o(1)&=&(J'_\varepsilon(u_n), u_n\eta^2_R)\\
&\geq& \frac{\varepsilon^2}{2}\int_{\mathbb R^3}|\nabla u_n|^2\eta_R^2 dx+\varepsilon^2\int_{\mathbb R^3} \Big
(\frac{1}{2}|\nabla u_n|^2-H(x)u_n^2\Big )\eta^2_R dx\\&+& \int_{\mathbb R^3}\phi_{u_n}u_n^2
\eta_R^2\rho(x)dx+(1-\mu)\int_{\mathbb R^3}V(x)u_n^2 \eta_R^2dx+ O
\Big(\frac{1}{R}\Big)\\&\geq&\frac{\varepsilon^2}{2}\int_{\mathbb R^3}|\nabla u_n|^2\eta_R^2 dx+ \int_{\mathbb
R^3}\phi_{u_n}u_n^2 \eta_R^2\rho(x)dx+ (1-\mu)\int_{\mathbb R^3}V(x)u_n^2 \eta_R^2dx+ O \Big(\frac{1}{R}\Big),
\end{eqnarray*}
as $R\rightarrow \infty$. Hence, taking $B:=\{x\in\R^3\,:\, |x|\leq e^{2R}\},$ since all the terms are nonnegative, the
above estimates yield
statements $(a),(b)$ and, using (\ref{gradH}), statement $(c).$

In order to prove $(d)$ we use Cauchy-Schwarz inequality and (\ref{liebloss}), obtaining
$$\int_{{\mathbb {R}}^{3}}\phi_{u_{n}}(x)(u_{n})_-\,u\,\rho (x)\eta_R dx\leq
(D(u_n^2,(u_n)_-^2))^{1/2}\Big(D(u_n^2,u_n^2)D(u^2,u^2)\Big)^{1/4}\rightarrow 0, $$ since $D(u_n^2,u_n^2)$ is bounded and
$D(u_n^2,(u_n)_-^2)\rightarrow 0$ by Lemma \ref{negativepartSP}.

Finally we prove $(e).$ We have, for any $R>0,$
\begin{align*}
o(1) &= (J'_\varepsilon(u_n), u\eta_R)\\
&\geq <u_n, u \eta_R>_{H_{V,\varepsilon}}+\int_{{\mathbb {R}}^{3}}\phi_{u_{n}}(x)(u_{n})_+\,u\,\rho (x)\eta_R
dx \\
&\quad -\int_{{\mathbb {R}}^{3}}\phi_{u_{n}}(x)(u_{n})_-\,u\,\rho (x)\eta_R dx - \int_{{\mathbb{R}}^{3}}(\varepsilon^2
H(x)+\mu V(x))(u_n)_+ u \eta_Rdx, \quad n\rightarrow \infty.
\end{align*}
Notice that, by weak convergence,
we have, for any $R$,
$$<u_n, u \eta_R>_{H_{V,\varepsilon}}\rightarrow <u, u \eta_R>_{H_{V,\varepsilon}}, $$
and
$$\int_{{\mathbb{R}}^{3}}(\varepsilon^2 H(x)+\mu V(x))(u_n)_+ u \eta_Rdx\rightarrow \int_{{\mathbb {R}}^{3}}(\varepsilon^2 H(x)+\mu V(x))u^2\eta_Rdx.$$
Now fix $\alpha>0$ small and take $R_\alpha$ such that for all $R> R_\alpha$ we have
$$\int_{{\mathbb {R}}^{3}}(\varepsilon^2 H(x)+\mu V(x))u^2 \eta_Rdx\leq\alpha$$
and, by $(d)$,
$$\int_{{\mathbb{R}}^{3}}\phi_{u_{n}}(x)(u_{n})_-\,u\,\rho (x)\eta_R dx\leq\alpha.$$
Arguing as for the estimate (\ref{har}), we can choose $R_\alpha$ large enough such that
$$\Big |\int_{{\mathbb {R}}^{3}}u\nabla u \nabla \eta_R dx\Big |\leq \frac{C}{R}<\alpha,$$ for every $R >R_\alpha.$ Hence, writing $\nabla(\eta_R \,u)=\eta_R\nabla u+u\nabla \eta_R,$ the term $<u,u\eta_R>_{H_{V,\varepsilon}}$ is the sum of a positive term plus a small term. Therefore, we obtain
$$o(1)+3 \alpha \geq \int_{{\mathbb{R}}^{3}}\phi_{u_{n}}(x)(u_{n})_+\,u\,\rho (x)\eta_R dx$$
and claim $(e)$ follows. This concludes the proof.
\end{proof}
Arguing as in the above lemmas we have

\begin{Lem} \label{sing2}
Under the assumptions on $\rho,V,K$ given in Theorem \ref{Th:main}, let $(u_n)$ be as in the above lemma. Then for all $\delta>0,$ there exists a ball $B(0)\subset \R^3,$
 such that
\end{Lem}
\begin{eqnarray*}
  &(a)& \quad \quad \limsup_ {n\rightarrow \infty}\int_{B(0)}\phi_{u_{n}}(x)u^2_{n}\rho (x)dx<\delta,\\
  &(b)&\quad \quad \limsup_ {n\rightarrow \infty}\Big |\int_{B(0)}\phi_{u_{n}}(x)(u_{n})_-u \rho (x)dx\Big
|<\delta,\\
 &(c)&\quad \quad \limsup_ {n\rightarrow \infty}\Big |\int_{B(0)}\phi_{u_{n}}(x)(u_{n})_+u \rho (x)dx\Big
|<\delta.
 \end{eqnarray*}

\begin{Lem}\label{HVVV2}
Let $(u_n)_n$  be  as in Lemma \ref{complem}. Then, passing if necessary to a subsequence, we have
$$\left\|u_{n}\right\|^{2}_{H_{V,\varepsilon}}\rightarrow  \left\|u\right\|^{2}_{H_{V,\varepsilon}}.$$
\end{Lem}
\begin{proof}
Since $u_n\rightharpoonup u$ in $H_{V,\varepsilon},$ for some subsequence, we have $$o(1)=
(J'_\varepsilon
(u_{n}),u_{n}-u)=\left\|u_{n}\right\|^{2}_{H_{V,\varepsilon}}-\left\|u\right\|^{2}_{H_{V,\varepsilon}}+o(1)$$
\begin{equation}\label{asta2}
+\int_{{\mathbb {R}}^{3}}\phi_{u_{n}}(x)u_{n}(u_{n}-u)\rho(x)dx+\int_{{\mathbb {R}}^{3}}g_\varepsilon
(x,u_n)(u_{n}-u)dx.
\end{equation}
We show that $$A_n:=\int_{{\mathbb {R}}^{3}}g_\varepsilon (x,u_n)(u_{n}-u)dx\rightarrow 0$$
and $$B_n:=\int_{{\mathbb {R}}^{3}}\phi_{u_{n}}(x)u_{n}(u_{n}-u)\rho(x)dx\rightarrow 0.$$
 Observe that $$A_n:=\int_{\Lambda}...+\int_{B\setminus \Lambda}...+\int_{\mathbb R^3\setminus B}...,$$
 for some large ball $B$ containing $\Lambda.$
Since $(u_n)_n$ is bounded in $E$ and $E$ is compactly embedded in $L^q(\Lambda)$ for all $q>1,$
passing to a subsequence, we can assume $u_n\rightarrow u$ in
$L^{p+1}(\Lambda)$. As a consequence, passing if necessary to a subsequence, we have $|g_\varepsilon (x,u_n)|<g(x)$ for
some $g\in L^{\frac{p+1}{p}}(\Lambda).$  Using H\"older inequality and dominated convergence theorem, we have
$$\int_{\Lambda}g_\varepsilon (x,u_n)(u_{n}-u)dx\rightarrow 0.$$
In the same way, using the compact embedding
$E\hookrightarrow L^2_{\varepsilon^2 H+\mu V}(B\setminus \Lambda)$, it follows that
$$\int_{B\setminus\Lambda}g_\varepsilon(x,u_n)(u_{n}-u)dx\rightarrow 0.$$
Finally, taking $B$ large and using $(b),(c)$ in Lemma \ref{complem}, we have
$$\int_{\mathbb R^3\setminus B}g_\varepsilon (x,u_n)(u_{n}-u)dx\rightarrow 0,$$
hence $A_n\rightarrow 0.$

In order to prove $B_n\rightarrow 0$, we use a similar splitting argument.
Fix $\delta>0.$
\begin{eqnarray*}
B_n &=&\int_{B(0)}... +\int_{B\setminus B(0)}...+\int_{\R^3\setminus
B}...\\ &=&I_{1,n}+I_{2,n}+I_{3,n},
\end{eqnarray*}
Now we choose $B$ such that, using  Lemma \ref{complem}, we have $$|I_{3,n}|<\delta.$$
Shrinking the ball $B(0)$ if necessary, we infer from Lemma \ref{sing2} that $|I_{1,n}|<\delta$.

Next, we estimate $I_{2,n}$ as follows.
By H\"older and Sobolev inequalities, we have
\begin{multline}
\int_{B\setminus B(0)}\phi_{u_n}(x)|u_{n}(u_{n}-u)|\rho(x)dx\leq \\
C\|\rho\|_{L^{\infty}(B\setminus B(0))}\|\phi_{u_n}\|_{\mathcal D^{1,2}(\mathbb {R}^3)}\|u_{n}(u_{n}-u)\|_{L^{6/5}(B\setminus B(0))}.
\end{multline}
Due to the weak convergence in $E$,
$\|\phi_{u_n}\|_{\mathcal D^{1,2}(\mathbb {R}^3 )}$ is
 bounded, and therefore, by compactness, we get
$$|I_{2,n}|<\delta.$$
This concludes the proof.
 \end{proof}

\begin{proof}[Proof of Theorem \ref{expena}]

By Lemma \ref{MMMPPP} and Lemma \ref{negativepartSP}, there exists a bounded P-S  sequence $(u_n)_n$ at the
minimax level $c_\varepsilon,$ such that $u_n\rightharpoonup u $ in $E.$  We are going to prove that, passing if
necessary to a subsequence,
\begin{eqnarray*} &i)& \quad \quad J_\varepsilon(u_n)\rightarrow J_\varepsilon(u), \\
&ii)& \quad \quad J'_\varepsilon(u)=0.
\end{eqnarray*}
This will imply the existence of a nontrivial solution $u.$

In order to prove  $i),$ notice that,  by Lemma
\ref{HVVV2}, it is enough to show that

$$\int_{\R^3}G_\varepsilon(x,u_n)dx\rightarrow \int_{\R^N}G_\varepsilon(x,u)dx$$
and
$$\int_{{\mathbb {R}}^{3}}\phi_{u_{n}}(x)u^2_{n}\rho(x)dx\rightarrow
\int_{{\mathbb {R}}^{3}}\phi_{u}(x)u^2\rho(x)dx.$$
In order to prove the former limits, we can argue as for the terms involving $g_\varepsilon$ in Lemma \ref{HVVV2},  splitting the integral $$\int_{\R^3}|G_\varepsilon(x,u_n)- G_\varepsilon(x,u)|dx= \int_{\Lambda}...+\int_{B\setminus \Lambda}...+\int_{\mathbb R^3\setminus B}... \,.$$ We can assume $u_n\rightarrow u$ in
$L^{p+1}(\Lambda)$ and almost everywhere. Fix $\delta >0.$  By using the compact embedding
$E\hookrightarrow L^q(\Lambda)$ which holds for all $q>1,$ the dominated convergence theorem yields 
$$\int_{\Lambda}|G_\varepsilon(x,u_n)- G_\varepsilon(x,u)|dx <\delta$$
for large $n$ and, in the same fashion, using property $(g_3)$ and the compact embedding of $E$ in
$ L^2_{\varepsilon^2 H+\mu V}(B\setminus \Lambda)$, it follows that
$$\int_{B\setminus \Lambda}|G_\varepsilon(x,u_n)- G_\varepsilon(x,u)|dx <\delta$$ 
for some subsequence, taking $n$ larger if necessary.
Finally, observe that, by $(b),(c)$ of Lemma \ref{complem}, there exists $B$ large enough, such that for $n$ large enough,
$$\int_{\mathbb R^3\setminus B}|G_\varepsilon(x,u_n)- G_\varepsilon(x,u)|dx<\delta.$$

Now we prove the second limit. We compute
 \begin{align*}
&\left|\int_{{\mathbb {R}}^{3}}[\phi_{u_{n}}(x)u^2_{n}\rho(x)-\phi_{u}(x)u^2\rho(x)]dx\right| \\
&\qquad \leq
\Big|\int_{B(0)}... \Big|+\Big|\int_{B\setminus B(0)}...\Big|+\Big|\int_{\R^3\setminus B}...\Big|\,\\
&\qquad = J_{1,n}+J_{2,n}+J_{3,n}.
\end{align*}
Fix $\delta >0$. Arguing as in the proof of Lemma \ref{HVVV2}, we can take $B$ large enough in such a way that
$$\int_{\R^3\setminus B}\phi_{u}(x)u^2\rho(x)dx<\delta,$$
yielding, with Lemma \ref{complem},
$$J_{3,n}\leq \Big|\int_{\R^3\setminus B}\phi_{u_{n}}(x)u^2_{n}\rho(x)dx\Big|+\Big|\int_{\R^3\setminus
B}\phi_{u}(x)u^2\rho(x)dx\Big|<2 \delta.$$
Since we can shrink $B(0)$ so that
$$\int_{B(0)}\phi_{u}(x)u^2\rho(x)dx<\delta,$$
we deduce from Lemma \ref{sing2} that
$$J_{1,n}\leq \Big|\int_{B(0)}\phi_{u_{n}}(x)u^2_{n}\rho(x)dx\Big|+\Big|\int_{B(0)}\phi_{u}(x)u^2\rho(x)dx\Big|<2 \delta.$$
Now, with by now familiar arguments, we can estimate $J_{2,n}$.
Indeed, using H\"older and Sobolev inequalities, we have
\begin{multline}
\int_{B\setminus B(0)}|\phi_{u_{n}}(x)u^2_{n}-\phi_{u}(x)u^2|\rho(x)dx\leq  \\
\int_{B\setminus B(0)}|\phi_{u_{n}}(x)u^2_{n}-\phi_{u_n}(x)u^2|\rho(x)dx+\int_{B\setminus B(0)}|\phi_{u_{n}}(x)u^2-\phi_{u}(x)u^2|\rho(x)dx\leq \\
 C \|\rho\|_{L^{\infty}(B\setminus B(0))}\|\phi_{u_n}\|_{\mathcal D^{1,2}(\mathbb {R}^3
)}\|u^2_{n}-u^2\|_{L^{6/5}(B\setminus B(0))} + 
\\
+\|\rho\|_{L^{\infty}(B\setminus B(0))}\int_{B\setminus B(0)}|\phi_{u_{n}}(x)u^2-\phi_{u}(x)u^2|dx. \label {secondo2}
\end{multline}
By compactness, we infer that $\|u^2_{n}-u^2\|_{L^{6/5}(B\setminus B(0))}\rightarrow 0$, while
$(\phi_{u_n})_n$ is bounded in $\mathcal D^{1,2}(\R^3),$ hence the first term in (\ref{secondo2}) goes to zero. On the other hand, since $\phi_{u_n}\rightharpoonup \phi_u$ in $\mathcal D^{1,2}(\R^3),$ we have $\phi_{u_n}\rightarrow
\phi_u$ strongly in $L^d(B\setminus B(0))$ for any $d<2^*.$ Hence, H\"older inequality implies the last term in
(\ref{secondo2}) goes also to zero. As a consequence $J_{2,n}\leq\delta$ and this yields $i).$

The proof of $ii)$ is rather standard, using the weak convergence in $E$ and the same splitting arguments. The maximum principle implies $u>0$ on $\mathbb R^3.$ This completes the proof of the theorem.

\end{proof}

\section{Asymptotics of solutions}
In order to show that the solution $u_{\varepsilon}$ found in Theorem \ref{expena} satisfies, for $\varepsilon$ small
enough, the original problem and concentrates around a circle, we need to study the asymptotic behaviour of
$u_{\varepsilon}$ as $\varepsilon \to 0$. Since many arguments are similar to the ones in
\cite{BDCVS}, we only stress the differences with these. We begin with an energy
estimate.

\begin{Prop}[Upper estimate of the critical value]\label{estim:inf}
 Suppose that the assumptions of Theorem \ref{expena} are satisfied. For $\varepsilon$ small enough, the critical value
$c_{\varepsilon}$ defined in \eqref{ceps} satisfies
 \[\label{estimc}
 c_{\varepsilon} \leq \varepsilon^{2} \bigl( \pi \inf_{\Lambda \cap \pi} \mathcal{M} + o(1) \bigr)\
\text{as}\ \varepsilon \rightarrow 0.
 \]
 Moreover, the solution $u_{\varepsilon}$ of \eqref{NLSPropen} found in Theorem $\ref{expena}$ satisfies, for some $C>0$,
\begin{align}\label{boundsol}
 \norm{u_{\varepsilon}}_{H_{{V,\varepsilon}}}^2 \leq C\varepsilon^{2}
\end{align}
and
\begin{align}\label{boundgradphi}
 \int_{\R^3} \abs{\nabla \phi_{u_{\varepsilon}}}^2\: dx =
 \int_{\mathbb{R}^3}\int_{\mathbb{R}^3}\frac{u_{\varepsilon}^2(x)u_{\varepsilon}^2(y)\rho(x)\rho(y)}{|x-y|
}\: dx\: dy \leq C \varepsilon^{2}.
\end{align}
\end{Prop}
\begin{proof}
 Take $x_0=(0, x''_0) \in \Lambda \cap \pi$ such that $\mathcal{M}(x_0)=\inf_{\Lambda
\cap \pi} \mathcal{M}$.
Denote by $I_0$ the functional defined by \eqref{defIab} with $a = V(x_0)$ and $b = K(x_0)$ and let $c_0 :=
\mathcal{E}(V(x_0),K(x_0))$. From \eqref{EabMP}, we infer that for every $\delta > 0$, there exists a continuous path
$\gamma_{\delta} : [0,1] \to H^1(\R^{2})$ such that $\gamma_{\delta}(0) = 0$, $I_0(\gamma_{\delta}(1)) < 0$ and
\[
 c_0 \leq \max_{t \in [0,1]} I_0(\gamma_{\delta}(t)) \leq c_0 + \delta.
\]
Let $\eta \in \mathcal{D}\left( \R^{2}_+ \right)$ be a cut-off function with support in $\Lambda$
such that $0 \leq \eta \leq 1$, $\eta = 1$ in a neighbourhood of $(0,\abs{x''_0})$ and $\norm{\nabla\eta}_{\infty}
\leq C$. We consider the path
\begin{align*}
 \bar{\gamma}_{\delta}(t) : x \to \eta(x',\abs{x''}) \gamma_{\delta}(t)\left(\frac{x'}{\varepsilon},
\frac{\abs{x''}-\abs{x''_0}}{\varepsilon}\right).
\end{align*}
Setting
\begin{align*}
 \bar{\gamma}_{\delta}(t)(x',x'') =: v_t\left(\frac{x'}{\varepsilon}, \frac{\abs{x''}-\abs{x''_0}}{\varepsilon}\right),
\end{align*}
we compute, by a change of variable,
\begin{multline*}
 \frac{1}{2}\int_{\R^3}(\varepsilon^2 \abs{\nabla \bar{\gamma}_{\delta}(t)}^{2}+V(x)\bar{\gamma}_{\delta}(t)^{2}) dx -
\int_{\R^3} G_\varepsilon (x,\bar{\gamma}_{\delta}(t)) dx \\
= \frac{\pi}{2} \int_{\R} \int_{-\frac{\abs{x''_0}}{\varepsilon}}^{\infty} \left( \abs{\nabla v_t}^2 +
V(\varepsilon y',\varepsilon \sigma + \abs{x''_0}) v_t^2 \right) (\varepsilon \sigma +\abs{x''_0}) \varepsilon\:
d\sigma\:
\varepsilon dy'\\ - \pi \int_{\R} \int_{-\frac{\abs{x''_0}}{\varepsilon}}^{\infty} G(\varepsilon
y',\varepsilon \sigma + \abs{x''_0},v_t) (\varepsilon \sigma + \abs{x''_0}) \varepsilon\: d\sigma\:
\varepsilon dy'.
\end{multline*}
The boundedness of $\rho$ in $\Lambda$ and Hardy-Littlewood-Sobolev inequality leads to
\begin{align*}
 &\int_{\R^3}\int_{\R^3} \frac{\bar{\gamma}_{\delta}(t)(x)^2\bar{\gamma}_{\delta}(t)(y)^2\rho(x)\rho(y)}{|x-y|}\:
dx\: dy \\
&\quad \leq \norm{\rho}_{L^{\infty}(\Lambda)}^2 \int_{\R^{2}}\int_{\R^{2}} \frac{v_t^2(x',\sigma) v_t^2(y',\tau)}
{\varepsilon \left(|x'-y'|^2+\abs{\sigma-\tau}^2\right)^{\frac{1}{2}}} \epsilon^{4}\: dx'\: dy'\: d\sigma\:
d\tau \\
&\quad \leq C \varepsilon^{3} \norm{v_t}_{L^{2}(\R^{2})}^4 \\
&\quad = o\left( \varepsilon^{2} \right).
\end{align*}
For $\varepsilon$ small enough, we obtain
\begin{align}\label{estimc1}
 \varepsilon^{-2} J_{\varepsilon}(\bar{\gamma}_{\delta}(t)) \leq \omega_k \abs{x''_0}^{k} I_0(\gamma_{\delta}(t)) +
o(1).
\end{align}
It follows that for $\varepsilon$ small enough, $\bar{\gamma}_{\delta}$ belongs to the class of paths
$\Gamma_{\varepsilon}$ defined by
\eqref{Gammaeps}. We deduce from \eqref{ceps} that
\[
\begin{split}
 \varepsilon^{-2} c_{\varepsilon} &\leq \max_{t \in [0,1]} \varepsilon^{-2}
J_{\varepsilon}(\bar{\gamma}_{\delta}(t))  \\
&\leq \pi \abs{x''_0} \max_{t \in [0,1]} I_0(\gamma_{\delta}(t)) + o(1)  \\
&\leq \pi \abs{x''_0} (c_0+\delta) + o(1).
\end{split}
\]
Since $\delta > 0$ is arbitrary and $\abs{x''_0} c_0 = \mathcal{M}(x_0)$, the first statement is established.
The second statement is proved by a computation similar to \eqref{boundedPS}.
\end{proof}

\begin{Prop}[No uniform convergence to $0$ in $\Lambda$]\label{lemma:no0}
Suppose that the assumptions of Theorem $\ref{expena}$ are satisfied and let $(u_{\varepsilon})_{\varepsilon} \subset E$
be positive solutions of \eqref{NLSPropen} obtained in Theorem $\ref{expena}$. Then there exists $\delta > 0$ such that
\begin{align*}
 \norm{u_{\varepsilon}}_{L^{\infty}(\Lambda)} \geq \delta.
\end{align*}
\end{Prop}
\begin{proof}
See \cite[Proposition 4.2]{BDCVS}.
\end{proof}

 By the symmetry imposed on the solution $u_{\varepsilon}$, one can write $u_{\varepsilon}(x',x'') =
\Tilde{u}_{\varepsilon}(x',\abs{x''})$ with $\Tilde{u}_{\varepsilon} : \R^{2}_+ \to \R$. Notice that
$\phi_{u_{\varepsilon}}$ has
the same symmetry as $u_{\varepsilon}$, i.e. for all $R \in \mathbf{O}(3)$ such that $R({d}) = {d}$,
 $\phi_{u_{\varepsilon}} \circ R = \phi_{u_{\varepsilon}}$. This follows easily from the representation formula
\eqref{formula:phi}.

Since the
$H_{V,\varepsilon}$-norm of $u_{\varepsilon}$ is of the order $\varepsilon$, it
is natural to rescale $\Tilde{u}_{\varepsilon}(x',\abs{x''})$ as $\Tilde{u}_{\varepsilon}(x'_{\varepsilon} + \varepsilon
y', \abs{x''_{\varepsilon}} + \varepsilon \abs{y''})$ around a well-chosen family of points $x_{\varepsilon} =
(x'_{\varepsilon},x''_{\varepsilon}) \in \R^3$. The next lemma shows that the sequences of rescaled solutions converge,
up to a subsequence, in $C^1_{\mathrm{loc}}(\R^{2})$ to a function $v \in H^1(\R^{2})$.

\begin{Lem}
\label{convergentSubsequence}
Suppose that the assumptions of Theorem \ref{expena} are satisfied.
Let $(u_{\varepsilon})_{\varepsilon} \subset E$ be
positive solutions of \eqref{NLSPropen} found in Theorem \ref{expena}, $(\varepsilon_n)_n \subset \R^+$ and $(x_n)_n
\subset \R^3$ be sequences such that $\varepsilon_n \to 0$ and $x_n = (x'_n,x''_n) \to \bar{x} =
(\bar{x}',\bar{x}'') \in \Bar{\Lambda}$ as $n\to \infty$. Set
\[
  \Omega_n := \R \times \Bigl] -\frac{\abs{x''_n}}{\varepsilon_n}, +\infty \Bigr[
\]
and let  $v_n : \Omega_n \to \R$ be defined by
\begin{align}\label{def:vn}
 v_n(y, z) := \Tilde{u}_{\varepsilon_n}(x'_{n} + \varepsilon_n y, \abs{x''_{n}} + \varepsilon_n z),
\end{align}
where $\Tilde{u}_{\varepsilon_n} : \R^{2}_+ \to \R$ is such that $u_{\varepsilon_n}(x',x'')
=\Tilde{u}_{\varepsilon_n}(x', \abs{x''})$. Then, there exists $v \in H^1(\R^{2})$ such that, along a subsequence that
we still denote by $(v_n)_n$,
\begin{align*}
 v_n \stackrel{C^1_{\mathrm{loc}}(\R^{2})}{\longrightarrow} v.
\end{align*}
Moreover, $v$ is a solution of the equation
\begin{align*}
-\Delta v + V(\bar{x}) v = K(\bar{x}) v^p, \  \bar x\in \R^{2}.
\end{align*}
\end{Lem}
\begin{proof}
We infer from Proposition \ref{estim:inf} that for all $n \in \N$,
\begin{align}\label{boundvn}
 \int_{\Omega_n} \left( \abs{\nabla v_n(y,z)}^2 + V(x'_{n} + \varepsilon_n y, \abs{x''_{n}} + \varepsilon_n z)
\abs{v_n(y,z)}^2 \right)\: dy\, dz \leq C,
\end{align}
with $C>0$ independent of $n$.

Observe also that each $v_n$ solves the equation
\begin{align}\label{conv:eqn}
 -\Delta v_n - \frac{\varepsilon_n}{z} \Dp{v_n}{z} + V(x'_{n} + \varepsilon_n y, \abs{x''_{n}} + \varepsilon_n z) v_n
 + \rho(x'_{n} + \varepsilon_n y, \abs{x''_{n}} + \varepsilon_n z) \tilde{\phi}_n v_n \nonumber \\ \hspace{4cm} =
g_{\varepsilon_n} (x'_{n} + \varepsilon_n y, \abs{x''_{n}} + \varepsilon_n z, v_n), \,\,\,\,\,\, x \in \R^{2},
\end{align}
where we have set $\tilde{\phi}_n(y,z) = \phi_{u_{\varepsilon_n}}(x'_{n} + \varepsilon_n y, \abs{x''_{n}} +
\varepsilon_n z)$.
As a consequence of \eqref{boundgradphi}, the sequence $(\tilde{\phi}_n)_n$ converges to zero in
$\mathcal{D}^{1,2}(\mathbb{R}^{2})$. It follows then from H\"older inequality that the term $(\rho(x'_{n} +
\varepsilon_n y,
\abs{x''_{n}} + \varepsilon_n z) \tilde{\phi}_n v_n)_n$ is bounded in $L^2_{\text{loc}}\left(\R^{2}\setminus
\left\{0\right\}\right)$.

Define a cut-off function $\eta_R \in \mathcal{D}(\R^{2})$ such that $0\leq \eta_R \leq 1$, $\eta_R(x) = 1$ if
$\abs{x} \leq R/2$, $\eta_R(x)=0$ if $\abs{x}\geq R$ and $\norm{\nabla\eta_R}_{\infty} \leq C/R$ for some $C>0$. Choose
$(R_n)_n$ such that $R_n\to\infty$ and $\varepsilon_n R_n \to 0$. Since $\Bar{x} \in \Lambda$ and $\Bar{\Lambda} \cap
d=\emptyset$, one has $\varepsilon_n R_n \leq \abs{x''_{n}}$ if $n$ is large enough.
Define $w_n \in H^1_{\mathrm{loc}}(\R^{2})$ by
\[
w_n(y) := \eta_{R_n}(y) v_n(y).
\]

It was shown in \cite[Lemma 4.3]{BDCVS} that \eqref{boundvn} implies that $(w_n)_n$ is bounded in $H^1(\R^{2})$.
 Since $w_n$ solves equation \eqref{conv:eqn} on $B(0,R_n)$ for all $n$, classical regularity estimates yield that for
every $R>0$ and every $q>1$,
\begin{align}\label{conv:estimW2q}
 \sup_{n\in \N} \norm{v_n}_{W^{2,q}(B(0,R))} < \infty.
\end{align}

Up to a subsequence, we can now assume that $(w_n)_n$ converges weakly in $H^1(\R^{2})$ to some function $v \in
H^1(\R^{2})$. By \eqref{conv:estimW2q}, for every
compact set $K \subset \R^{2}$, $w_n$ converges to $v$ in $C^1(K)$. Moreover, for $n$ large enough,
$w_n = v_n$ in $K$ so that $v_n\to v$ in $C^1(K)$.
\end{proof}

For $x,y \in \R^3$, denote by
\begin{align*}
 d_{d}(x,y) := \sqrt{\abs{x'-y'}^2+ \left( \abs{x''}-\abs{y''}\right) ^2}.
\end{align*}
the distance between the circles centered at $x'$ and $y'$, and of radius
$\abs{x''}$ and $\abs{y''}$ respectively. 
We denote by $B_{d}$ the balls for the distance $d_{d}$, i.e.,
\[
 B_{d}(x,r)=\{ y \in \R^3 : d_{d}(x, y) < r\}.
\]
We are now going to estimate from below the critical value $c_{\varepsilon}$. In the next two lemmas we estimate the
action
respectively inside and outside neighbourhoods of points.

\begin{Lem}
Suppose that the assumptions of Theorem \ref{expena} are satisfied. Let $u_{\varepsilon}\in E$ be
positive solutions of \eqref{NLSPropen} found in Theorem \ref{expena}, $(\varepsilon_n)_n \subset \R^+$ and $(x_n)_n
\subset
\R^n$ be sequences such that $\varepsilon_n \to 0$ and $x_n = (x'_n,x''_n) \to \bar{x} =
(\bar{x}',\bar{x}'') \in \Bar{\Lambda}$ as $n\to \infty$. If
\begin{align}\label{estim:sup1}
 \liminf_{n\to\infty} u_{\varepsilon_n}(x_n) > 0,
\end{align}
then we have, up to a subsequence,
 \begin{align*}
 \liminf_{R\to\infty}\liminf_{n\to\infty} \varepsilon_n^{-2} \left( \int_{T_n(R)}  \frac{1}{2} \left(
\varepsilon_n^2
\abs{\nabla u_{\varepsilon_n}}^2 + V u_{\varepsilon_n}^2 \right) - G_{\varepsilon_n}(x,u_{\varepsilon_n}) \right)
\geq \pi \mathcal{M}(\bar{x}),
 \end{align*}
where $T_n(R) := B_{d}(x_n,\varepsilon_n R)$.
\end{Lem}
\begin{proof}
 The proof is the same as the one of Lemma $4.4$ in \cite{BDCVS}. Indeed, the Poisson term is positive and
the equation satisfied by the limit of the sequence of rescaled solutions is the same as in \cite{BDCVS}.
\end{proof}

\begin{Lem}
 Suppose that the assumptions of Theorem \ref{expena} are satisfied. Let $u_{\varepsilon}\in E$ be
positive solutions of \eqref{NLSPropen} found in Theorem \ref{expena}, $(\varepsilon_n)_n \subset \R^+$ and $(x^i_n)_n
\subset \R^3$ be sequences such that $\varepsilon_n \to 0$ and for $1\leq i \leq M$, $x^i_n\to
\bar{x}^i \in \Bar{\Lambda}$ as $n\to \infty$. Then, up to a subsequence, we have
\begin{align*}
  \liminf_{R\to\infty}\liminf_{n\to\infty} \varepsilon_n^{-2} \left( \int_{\R^3 \setminus \mathcal{T}_n(R)}
\frac{1}{2} \left( \varepsilon_n^2 \abs{\nabla u_{\varepsilon_n}}^2 + V u_{\varepsilon_n}^2 \right) -
G_{\varepsilon_n}(x,u_{\varepsilon_n}) \right) \geq 0,
\end{align*}
where $\mathcal{T}_n(R) := \bigcup_{i=1}^K B_{d}(x^i_n,\varepsilon_n R)$.
\end{Lem}
\begin{proof}
 Since the Poisson term is positive, the proof is the same as the one of Lemma $4.5$ in \cite{BDCVS}.
\end{proof}

\begin{Prop}[Lower estimate of the critical value]\label{estim:sup}
 Suppose that the assumptions of Theorem \ref{expena} are satisfied. Let $u_{\varepsilon}\in E$ be
positive solutions of \eqref{NLSPropen} found in Theorem \ref{expena}, $(\varepsilon_n)_n \subset \R^+$ and $(x^i_n)_n
\subset \R^3$ be sequences such that $\varepsilon_n \to 0$ and for $1\leq i \leq M$, $x^i_n\to
\bar{x}^i \in \Bar{\Lambda}$ as $n\to \infty$. If for every $1\leq i < j \leq M$, we have
\begin{align*}
 \limsup_{n\to\infty} \frac{d_{d}(x^i_n,x^j_n)}{\varepsilon_n} = \infty
\end{align*}
and if for every $1\leq i \leq M$,
\[
 \liminf_{n\to\infty} u_{\varepsilon_n}(x^i_n) > 0,
\]
then the critical value $c_{\varepsilon}$ defined in \eqref{ceps} satisfies
 \begin{align*}
 \liminf_{n\to\infty} \varepsilon_n^{-2} c_{\varepsilon_n} \geq \pi \sum_{i=1}^M \mathcal{M}(\bar{x}^i).
 \end{align*}
\end{Prop}
\begin{proof}
This is a consequence of the two previous lemmas and of the positivity of the Poisson term, see \cite[Proposition
16]{BVS} for the details.
 \end{proof}

 Now we can state a first concentration result. It will be completed in the next section by a decay estimate.
 \begin{Prop}[Uniform convergence to $0$ outside small balls]\label{Prop1}
 Suppose that the assumptions of Theorem \ref{expena} are satisfied and that $\Lambda$ satisfies the assumptions
\eqref{hypLambda1}-
 \eqref{hypLambda4}.
Let $(u_{\varepsilon})_{\varepsilon} \subset E$ be positive solutions of \eqref{NLSPropen} obtained in
Theorem \ref{expena}. If $(x_{\varepsilon})_{\varepsilon>0} \subset \Lambda$ is such that
\begin{align*}
 \liminf_{\varepsilon\to 0} u_{\varepsilon}(x_{\varepsilon}) > 0,
\end{align*}
then
\begin{enumerate}[(i)]
 \item $\lim_{\varepsilon\to 0} \mathcal{M}(x_{\varepsilon}) = \inf_{\Lambda \cap \pi} \mathcal{M}$,
 \item $\lim_{\varepsilon \to 0} \frac{\dist(x_\varepsilon, \pi)}{\varepsilon}=0$,
 \item $\liminf_{\varepsilon \to 0} d_{d}(x_{\varepsilon}, \partial \Lambda) > 0$,
 \item for every $\delta>0$, there exists $\varepsilon_0>0$ and $R>0$ such that, for every $\varepsilon \in
(0,\varepsilon_0)$,
\[
 \norm{u_{\varepsilon}}_{L^{\infty}\left(\Lambda\setminus B_{d}(x_{\varepsilon},\varepsilon R)\right)} \leq \delta.
\]
\end{enumerate}
\end{Prop}
\begin{proof}
All the assertions are proved by energy comparisons, using only Propositions \ref{estim:inf} and \ref{estim:sup}. A
detailed
proof can be found in \cite[Proposition 4.7]{BDCVS}.
\end{proof}

\section{Solution of the initial problem}\label{sect:init}

\subsection{Linear inequation outside small balls}
In this section we prove that for $\varepsilon$ small enough, the solutions of the penalized problem \eqref{NLSPropen}
are also
solutions of the initial problem \eqref{NLSP}. We follow the arguments of \cite{MVS} and \cite{BDCVS}. First we notice
that the solutions
of \eqref{NLSPropen} satisfy a linear inequation outside small balls. As observed in \cite[Theorem 5]{BoMerc}, the
function
$\phi_{u_{\varepsilon}}$ satisfies the estimate
\begin{align*}
\phi_{u_{\varepsilon}}(x) \geq \frac{C_{\varepsilon}}{C'_{\varepsilon}+\abs{x}},
\end{align*}
for some constants $C_{\varepsilon}, C'_{\varepsilon} > 0$.
Set
\begin{align*}
W_{\varepsilon}(x) := (1 - \mu) V(x) + \frac{C_{\varepsilon} \rho(x)}{C'_{\varepsilon}+\abs{x}}.
\end{align*}
\begin{Lem}\label{lemma:subsol}
 Suppose that the assumptions of Proposition \ref{Prop1} are satisfied and let $(u_{\varepsilon})_{\varepsilon} \subset
E$
 be positive solutions of \eqref{NLSPropen} found in Theorem \ref{expena} and
$(x_{\varepsilon})_{\varepsilon>0} \subset \Lambda$ be such that
\begin{align*}
 \liminf_{\varepsilon\to 0} u_{\varepsilon}(x_{\varepsilon}) > 0.
\end{align*}
Then there exist $R > 0$ and $\varepsilon_0>0$ such that for all $\varepsilon \in (0,\varepsilon_0)$,
\begin{align}\label{lowsol}
 -\varepsilon^2 \left( \Delta u_{\varepsilon} + H u_{\varepsilon} \right) + W_{\varepsilon} u_{\varepsilon} \leq 0
\hspace{0.5cm}
\text{in}\ \R^3\setminus B_d(x_{\varepsilon},\varepsilon R).
\end{align}
\end{Lem}
\begin{proof}
 Set
\begin{align*}
 \eta := \inf_{x\in \Lambda} \left( \frac{\mu V(x)}{K(x)} \right)^{\frac{1}{p-1}}.
\end{align*}
Since $V$ and $K$ are bounded positive continuous functions on $\Bar{\Lambda}$, $\eta > 0$.
By Proposition \ref{Prop1}, we can find $\varepsilon_0 > 0$ and $R >0$ such that for all $\varepsilon \in
(0,\varepsilon_0]$, one has
\begin{align*}
 u_{\varepsilon}(x) \leq \eta \hspace{0.5cm} \text{for all} \ x \in \Lambda \setminus
B_d(x_{\varepsilon},\varepsilon R).
\end{align*}
We conclude that
\begin{align*}
 -\varepsilon^2 \Delta u_{\varepsilon} + (1 - \mu)V u_{\varepsilon} + \rho \phi_{u_{\varepsilon}} u_{\varepsilon} \leq
-\varepsilon^2 \Delta u_{\varepsilon} + V u_{\varepsilon} + \rho \phi_{u_{\varepsilon}} u_{\varepsilon} - K
u_{\varepsilon}^p = 0 \hspace{0.5cm}
\end{align*}
in $\Lambda \setminus B_d(x_{\varepsilon},\varepsilon R)$.
The fact that $u_{\varepsilon}$ satisfies \eqref{lowsol} in $\R^3 \setminus\Lambda$ follows directly from the definition
of the penalized nonlinearity.
\end{proof}

This lemma suggests that we can compare the solution $u_{\varepsilon}$ with supersolutions of the operator
$-\varepsilon^2 \left( \Delta + H \right) + W_{\varepsilon}$ in order to obtain decay estimates of $u_{\varepsilon}$.

\subsection{Comparison functions}

In this section we recall results from \cite{BDCVS} about the comparison functions.
The next lemma provides a minimal positive solution of the operator $- \Delta - H$ in $\R^3\setminus
\Bar{\Lambda}$.

\begin{Lem}
\label{lemmaPsiepsilon}
For every $\varepsilon > 0$, there exists $\Psi_\varepsilon \in C^2\left((\R^3\setminus \lbrace 0 \rbrace) \setminus
\Lambda
\right)$ such that
\[
 \left\{ \begin{aligned}
          -\varepsilon^2(\Delta \Psi_\varepsilon + H \Psi_\varepsilon) + W_\varepsilon \Psi_\varepsilon &= 0 &
&\text{in}\ \R^3\setminus
\Bar{\Lambda}, \\
	\Psi_{\varepsilon} &= 1 & &\text{on}\ \partial \Lambda,
         \end{aligned}
 \right.
\]
and
\begin{equation}
\label{PsiepsilonFiniteIntegral}
 \int_{\R^3\setminus \Lambda} \biggl( \abs{\nabla \Psi_\varepsilon(x)}^2 + \frac{\abs{\Psi_\varepsilon(x)}^2}{\abs{x}^2}
\biggr)\: dx <
\infty.
\end{equation}
Moreover, there exists $C > 0$ such that, for every $x \in \R^3\setminus \Lambda$ and every $\varepsilon > 0$,
\begin{equation}
\label{DecayPsiepsilon}
 0 < \Psi_\varepsilon(x) \leq \frac{C}{1+\abs{x}}.
\end{equation}
\end{Lem}
\begin{proof}
See \cite[Lemma 5.2]{BDCVS}.
\end{proof}

As explained in \cite{MVS}, the estimate \eqref{DecayPsiepsilon} is the best one can hope for, at least if $W_{\varepsilon}$
decays
rapidly at infinity.
However, if $W_{\varepsilon}$ decays quadratically or subquadratically at infinity, we can improve
\eqref{DecayPsiepsilon}.

\begin{Lem}
\label{lemmaImprovedDecay}
Let $\Psi_\varepsilon$ be given by Lemma~\ref{lemmaPsiepsilon}.
\begin{enumerate}[(1)]
\item \label{Psiepsiloninfinityquadratic} If $\liminf_{\abs{x}\rightarrow\infty} W_{\varepsilon}(x)\abs{x}^{2} > 0$,
then there exist
$\lambda > 0$, $R > 0$ and $C > 0$ such that for every $\varepsilon > 0$ and $x \in \R^3\setminus B(0,R)$,
\[
  \Psi_\varepsilon (x) \leq
C\left(\frac{R}{\abs{x}}\right)^{\frac{1}{2}+\sqrt{\frac{1}{4} - \kappa
+ \frac{\lambda^2}{\varepsilon^2}}}.
\]
\item \label{PsiepsiloninfinitySuperquadratic} If $\liminf_{\abs{x}\rightarrow\infty} W_{\varepsilon}(x)\abs{x}^{\alpha}
> 0$ with
$\alpha < 2$, then there exist $\lambda > 0$, $R > 0$, $C > 0$ and $\varepsilon_0 > 0$ such that for every
$\varepsilon \in (0,\varepsilon_0)$ and $x \in \R^3\setminus B(0,R)$,
\[
  \Psi_\varepsilon (x) \leq C\exp \left(-\frac{\lambda}{\varepsilon}
\left(\abs{x}^\frac{2-\alpha}{2}-R^\frac{2-\alpha}{2}\right)\right).
\]
\item  If $\liminf_{\abs{x}\rightarrow 0} W_{\varepsilon}(x)\abs{x}^{2} > 0$, then there exist $\lambda > 0$, $r > 0$
and $C > 0$
such that for every $\varepsilon > 0$ and $x \in B(0,r)$,
\[
  \Psi_\varepsilon (x) \leq
C\left(\frac{\abs{x}}{r}\right)^{\sqrt{\frac{1}{4} -\kappa
+\frac{\lambda^2}{\varepsilon^2}}-\frac{1}{2}}.
\]
\item If $\liminf_{\abs{x}\rightarrow 0} W_{\varepsilon}(x)\abs{x}^{\alpha} > 0$ with $\alpha < 2$, then there exist
$\lambda >
0$, $R > 0$, $C > 0$ and $\varepsilon_0 > 0$ such that for every $\varepsilon \in (0,\varepsilon_0)$ and $x \in
B(0,r)$,
\[
  \Psi_\varepsilon (x) \leq C\exp \left(-\frac{\lambda}{\varepsilon}
\left(\abs{x}^{-\frac{\alpha-2}{2}}-r^{-\frac{\alpha-2}{2}}\right)\right).
\]
\end{enumerate}
\end{Lem}
\begin{proof}
See \cite[Lemma 5.3]{BDCVS}.
\end{proof}

Now we provide a comparison function that describes the exponential decay of $u_\varepsilon$ inside $\Lambda$.

\begin{Lem}\label{CompPeak}
Let $\Bar{x} \in \Lambda$ and $R > 0$ be such that
\begin{equation}
\label{condRxeps}
 B_{d}(\Bar{x},R) \subset \Lambda.
\end{equation}
Define
\begin{align}\label{def:Phieps}
\Phi_{\varepsilon}^{\Bar{x}}(x) := \cosh \left( \lambda \frac{R-d_{d}(x,\Bar{x})}{\varepsilon}\right).
\end{align}
There exists $\lambda > 0$ and $\varepsilon_0 > 0$ such that for every $\varepsilon \in (0,\varepsilon_0)$, one has
\begin{align*}
 -\varepsilon^2 \Delta \Phi^{\Bar{x}}_{\varepsilon} + W_{\varepsilon} \Phi^{\Bar{x}}_{\varepsilon} \geq 0 \ \text{in}\
B_d(\Bar{x},R).
\end{align*}
\end{Lem}
\begin{proof}
See \cite[Lemma 5.4]{BDCVS}.
\end{proof}

\begin{Lem}\label{lemma:barrier}
Let $(x_{\varepsilon})_{\varepsilon} \subset \Lambda$ be such that
 \begin{align*}
 \liminf_{\varepsilon \to 0} d_{d}(x_{\varepsilon}, \partial \Lambda) > 0
 \end{align*}
and $R > 0$.
Then, there exist $\varepsilon_0 > 0$ and a family of functions $(w_{\varepsilon})_{0<\varepsilon<\varepsilon_0} \subset
C^{1,1}_{\textnormal{loc}}((\R^3 \setminus \{0\}) \setminus B_d(x_{\varepsilon},\varepsilon R))$ such that for
all $\varepsilon \in (0,\varepsilon_0)$, one has
 \begin{enumerate}[(i)]
 \item $w_{\varepsilon}$ satisfies the inequation
\begin{align*}
 -\varepsilon^2 \left( \Delta + H\right)  w_{\varepsilon} + W_{\varepsilon} w_{\varepsilon} \geq 0 \ \text{in}\ \R^3
\setminus
B_d(x_{\varepsilon},\varepsilon R),
\end{align*}
 \item\label{barrierb} $\nabla w_{\varepsilon} \in L^2(\R^3\setminus B_d(x_{\varepsilon},\varepsilon R))$ and
$\frac{w_{\varepsilon}}{\abs{x}} \in L^2(\R^3\setminus B_d(x_{\varepsilon},\varepsilon R))$,
 \item \label{barrierc} $w_{\varepsilon}\geq 1$ on $\partial B_d(x_{\varepsilon},\varepsilon R)$,
\item \label{barrierd} for every $x \in B_d(x_{\varepsilon},\varepsilon R)$,
\[
 w_{\varepsilon}(x) \leq C \exp{\left( -\frac{\lambda}{\varepsilon} \frac{d_{d}(x, x_\varepsilon)}{1+d_{d}(x,
x_\varepsilon)}\right)}
\left( 1+\abs{x} \right)^{-1}, \hspace{0.5cm} x \in \R^3.
\]
 \end{enumerate}
Moreover,
\begin{enumerate}[(1)]
\item \label{Wepsiloninfinityquadratic} If $\liminf_{\abs{x}\rightarrow\infty} W_{\varepsilon}(x)\abs{x}^{2} > 0$, then
there exists
$\lambda > 0$, $\nu > 0$ and $C > 0$ such that for $\varepsilon > 0$ small enough,
\[
  w_\varepsilon (x) \leq C \exp{\left( -\frac{\lambda}{\varepsilon} \frac{d_{d}(x, x_\varepsilon)}{1+d_{d}(x,
x_\varepsilon)}\right)}
\left( 1+\abs{x} \right)^{-\frac{\nu}{\varepsilon}}.
\]
\item \label{WepsiloninfinitySuperquadratic} If $\liminf_{\abs{x}\rightarrow\infty} W_{\varepsilon}(x)\abs{x}^{\alpha} >
0$ with
$\alpha > 2$, then there exists $\lambda > 0$ and $C > 0$ such that for $\varepsilon > 0$ small enough,
\[
w_\varepsilon (x) \leq C\exp \left(-\frac{\lambda}{\varepsilon}\frac{d_{d}(x, x_\varepsilon)}{1+d_d(x,
x_\varepsilon)}(1+\abs{x})^\frac{2-\alpha}{2}\right).
\]
\item \label{Wepsilon0quadratic} If $\liminf_{\abs{x}\rightarrow 0} W_{\varepsilon}(x)\abs{x}^{2} > 0$, then there
exists $\lambda >
0$, $\nu > 0$ and $C > 0$ such that for $\varepsilon > 0$ small enough,
\[
  w_\varepsilon (x) \leq C \exp{\left( -\frac{\lambda}{\varepsilon} \frac{d_{d}(x, x_\varepsilon)}{1+d_{d}(x,
x_\varepsilon)}\right)}
\left( \frac{\abs{x}}{1+\abs{x}}  \right)^{\frac{\nu}{\varepsilon}}.
\]
\item \label{Wepsilon0Superquadratic}If $\liminf_{\abs{x}\rightarrow 0} W_{\varepsilon}(x)\abs{x}^{\alpha} > 0$ with
$\alpha >
2$,
then there exists $\lambda > 0$ and $C > 0$ such that for $\varepsilon > 0$ small enough,
\[
  w_\varepsilon (x) \leq C\exp \left(-\frac{\lambda}{\varepsilon}\frac{d_{d}(x, x_\varepsilon)}{1+d_d(x,
x_\varepsilon)}\left(\frac{\abs{x}}{1+\abs{x}}\right)^\frac{\alpha-2}{2}\right).
\]
\end{enumerate}
\end{Lem}
\begin{proof}
See \cite[Lemma 5.5]{BDCVS}.
\end{proof}

Thanks to the previous lemma, we obtain an upper bound on the solutions $(u_{\varepsilon})_{\varepsilon>0}$ of
\eqref{NLSPropen}.

\begin{Prop}\label{Prop:decay}
 Suppose that the assumptions of Proposition \ref{Prop1} are satisfied. Let $(u_{\varepsilon})_{\varepsilon>0} \subset
E$
 be the positive solutions of \eqref{NLSPropen} found in Theorem \ref{expena} and
$(x_{\varepsilon})_{\varepsilon>0} \subset \Lambda$ be such that
\begin{align*}
 \liminf_{\varepsilon\to 0} u_{\varepsilon}(x_{\varepsilon}) > 0.
\end{align*}
Then there exist $C>0$, $\lambda > 0$ and $\varepsilon_0>0$ such that for all $\varepsilon \in (0,\varepsilon_0)$,
\begin{align}\label{decay}
 u_{\varepsilon}(x) \leq C \exp{\left( -\frac{\lambda}{\varepsilon}
\frac{d(x,S^1_{\varepsilon})}{1+d(x,S^1_{\varepsilon})}\right)}
( 1+\abs{x} )^{-1}, \hspace{0.5cm} x \in \R^3.
\end{align}
Moreover, \eqref{Wepsiloninfinityquadratic}, \eqref{WepsiloninfinitySuperquadratic}, \eqref{Wepsilon0quadratic} and
\eqref{Wepsilon0Superquadratic} in Lemma~\ref{lemma:barrier} hold with $u_\varepsilon$ in place of $w_\varepsilon$.
\end{Prop}
\begin{proof}
See \cite[Lemma 5.6]{BDCVS}.
\end{proof}

\subsection{Solution of the original problem}
\begin{Prop}\label{exinit}
 Let $(u_{\varepsilon})_{\varepsilon>0} \subset E$ be the positive solutions of \eqref{NLSPropen} found in Theorem
\ref{expena}. Assume that one set $(\mathcal{G}_{\infty}^i)$ of growth conditions at infinity and one set
$(\mathcal{G}_{0}^j)$
of growth conditions at the origin hold. Then there exists $\varepsilon_0>0$ such that for all $\varepsilon \in
(0,\varepsilon_0)$, $u_{\varepsilon}$ solves the original problem \eqref{NLSP}.
\end{Prop}
\begin{proof}
 We infer from Lemma \ref{lemma:no0} that there exists a family of points $(x_{\varepsilon})_{\varepsilon>0} \subset
\Lambda$ such that
\begin{align*}
 \liminf_{\varepsilon\to 0} u_{\varepsilon}(x_{\varepsilon}) > 0.
\end{align*}
Assume that the assumptions $(\mathcal{G}_{\infty}^1)$ and $(\mathcal{G}_{0}^1)$ are satisfied.
By Proposition \ref{Prop:decay}, we obtain for $\varepsilon > 0$ small enough and $x \in \R^3\setminus \Lambda$,
\begin{align*}
 K(x) u_{\varepsilon}^{p-1} &\leq M (1+\abs{x})^{\sigma} \left( C e^{\frac{-\lambda}{\varepsilon}}
\frac{1}{\abs{x}} \right)^{p-1} \\
&\leq C e^{\frac{-\lambda}{\varepsilon}(p-1)} (1+\abs{x})^{-(p-1)+\sigma} \\
&\leq \frac{\varepsilon^2 \kappa}{\abs{x}^2 \left( (\log \abs{x})^2+1 \right)^{\frac{1+\beta}{2}}} = \varepsilon^2 H(x).
\end{align*}
We conclude by definition of the penalized nonlinearity $g_{\varepsilon}$ that
$g_{\varepsilon}(x,u_{\varepsilon}(x)) = K(x) u_{\varepsilon}^p(x)$, and hence $u_{\varepsilon}$ solves the original
problem \eqref{NLSP}. The other cases can be treated in a similar way.
\end{proof}

\begin{proof}[Proof of Theorem \ref{Th:main}]
We proved in Theorem \ref{expena} that, for any $\varepsilon > 0$, the penalized problem \eqref{NLSPropen} possesses a
solution $u_{\varepsilon}$. By Proposition \ref{exinit}, for $\varepsilon > 0$ small enough, $u_{\varepsilon}$ is a
solution
of the initial problem \eqref{NLSP}. The existence of a sequence $(x_{\varepsilon})_{\varepsilon>0} \subset
\Lambda$ such that
\begin{align*}
 \liminf_{\varepsilon\to 0} u_{\varepsilon}(x_{\varepsilon}) > 0
\end{align*}
follows from Lemma \ref{lemma:no0} and the concentration result follows from Proposition \ref{Prop1}. Finally,
Proposition \ref{Prop:decay} yields the decay estimate.
\end{proof}

\section{Remarks and further results}

\subsection{Concentration at points}
We can also obtain a result about solutions concentrating at points. In this case, the concentration function is given
by
\begin{align*}
 \mathcal{A}(x) = \left[V(x)\right]^{\frac{p+1}{p-1} - \frac{3}{2}}
\left[K(x)\right]^{\frac{-2}{p-1}}.
\end{align*}

\begin{Thm}\label{Th:points}
Let $3 < p < 5$, $V, K \in C(\R^3 \backslash \left\{ 0 \right\},\R^+)$,
$K\not\equiv 0$
and $\rho \in L^{3/2}_{\text{loc}}(\R^3)\cap L^{\infty}_{\text{loc}}\left(\R^3\setminus \left\{ 0 \right\}\right)$.
Assume that one set $(\mathcal{G}_{\infty}^i)$ of growth conditions at infinity and one set $(\mathcal{G}_{0}^j)$ of
growth conditions
at the origin hold.
Assume also that there exists a smooth open bounded set $\Lambda \subset \R^3$ such that
\begin{align}\label{hyp:Lambda}
 0 < \inf_{\Lambda} \mathcal{A} < \inf_{\partial \Lambda} \mathcal{A}.
\end{align}
 Then there exists $\varepsilon_0 > 0$ such that for every
$0 < \varepsilon < \varepsilon_0$, problem \eqref{NLSP} has at least one positive solution $u_{\varepsilon}$.
Moreover, for every $0 < \varepsilon < \varepsilon_0$, there exists $x_{\varepsilon} \in \Lambda$
such that $u_{\varepsilon}$ attains its maximum at $x_{\varepsilon}$,
\begin{align*}
 \liminf_{\varepsilon \to 0} u_{\varepsilon}(x_{\varepsilon}) &> 0, \\
 \lim_{\varepsilon \to 0} \mathcal{A} (x_{\varepsilon}) &= \inf_{\Lambda} \mathcal{A},
\end{align*}
and there exist $C>0$ and $\lambda > 0$ such that
\begin{align*}
 u_{\varepsilon}(x) &\leq C \exp{\left( -\frac{\lambda}{\varepsilon}
\frac{\abs{x-x_{\varepsilon}}}{1+\abs{x-x_{\varepsilon}}}\right) }
\left( 1+\abs{x-x_{\varepsilon}}^2 \right)^{\frac{-1}{2}}, &  \forall x \in \R^3.
\end{align*}
\end{Thm}

The proof of this theorem is similar to the proof of Theorem \ref{Th:main}, but simpler. Let us only sketch the proof.
First of all,
we impose no symmetry neither on the potentials nor on the solution. This makes the critical Sobolev  exponent to appear, in spite of what happens in the preceding results. We modify the problem in the same way as before and
we
search for a critical point of the functional $J_{\varepsilon}$ in the space $X$ defined by \eqref{def:X}. Theorem
\ref{expena} remains
true with the same proof.

The limiting problem associated to concentration at points is the problem
\begin{align*}
	 -\Delta u + au = bu^p \hspace{1cm} \text{in} \ \R^3.
\end{align*}
Let $x_0 \in \Lambda$ be a point such that $\mathcal{A}(x_0) = \inf_{\Lambda} \mathcal{A}$. We denote by $c_0$
the least energy critical value of the limiting problem with $a = V(x_0)$ and $b = K(x_0)$. As in Proposition
\ref{estim:inf}, we prove that the critical value $c_{\varepsilon}$ defined in \eqref{ceps} satisfies
 \[
 c_{\varepsilon} \leq \varepsilon^{3} \left( c_0 + o(1) \right)\
\text{as}\ \varepsilon \rightarrow 0.
 \]
Then we prove as before that the $L^{\infty}$-norm of $u_{\varepsilon}$ does not converge to $0$ in $\Lambda$ and
that the sequence of rescaled solutions converges in $C^1_{\text{loc}}$ to a solution of the limiting equation.
The analogous of Proposition \ref{estim:sup} is the following one.
\begin{Prop}
 Let $(\varepsilon_n)_n \subset \R^+$ and $(x^i_n)_n \subset \R^3$
be sequences such that $\varepsilon_n \to 0$ and for $1\leq i \leq M$, $x^i_n\to
\bar{x}^i \in \Bar{\Lambda}$ as $n\to \infty$. If for every $1\leq i < j \leq M$, we have
\begin{align*}
 \limsup_{n\to\infty} \frac{\abs{x^i_n-x^j_n}}{\varepsilon_n} = \infty
\end{align*}
and if for every $1\leq i \leq M$,
\[
 \liminf_{n\to\infty} u_{\varepsilon_n}(x^i_n) > 0,
\]
then
 \begin{align*}
 \liminf_{n\to\infty} \varepsilon_n^{-3} c_{\varepsilon_n} \geq \sum_{i=1}^M \mathcal{C}(\bar{x}^i),
 \end{align*}
where $\mathcal{C}(x)$ is the least energy critical value of the limiting problem with
$a = V(\bar{x})$ and $b = K(\bar{x})$.
\end{Prop}

The concentration result can be stated as follows.
 \begin{Prop}
 Suppose that $\Lambda$ satisfies \eqref{hyp:Lambda}.
If $(x_{\varepsilon})_{\varepsilon>0} \subset \Lambda$ is such that
\begin{align*}
 \liminf_{\varepsilon\to 0} u_{\varepsilon}(x_{\varepsilon}) > 0,
\end{align*}
then
\begin{enumerate}[(i)]
 \item $\lim_{\varepsilon\to 0} \mathcal{A}(x_{\varepsilon}) = \inf_{\Lambda} \mathcal{A}$,
 \item $\liminf_{\varepsilon \to 0} d(x_{\varepsilon}, \partial \Lambda) > 0$,
 \item for every $\delta>0$, there exists $\varepsilon_0>0$ and $R>0$ such that,
for every $\varepsilon \in (0,\varepsilon_0)$,
\[
 \norm{u_{\varepsilon}}_{L^{\infty}\left(\Lambda\setminus B(x_{\varepsilon},\varepsilon R)\right)} \leq \delta.
\]
\end{enumerate}
\end{Prop}

Finally the comparison arguments in order to get back to the original problem are the same as in section
\ref{sect:init}.

\subsection{Concentration on spheres}
Using the same method, we can prove the existence of solutions concentrating on a sphere for the following problem.
\begin{equation*}
\left\{
\begin{array}{l}
-\varepsilon^2 \Delta u+V(x)u+\rho(x)\phi u = K(x)u^p, \,\,\,\,\,\, x \in {\mathbb R}^{3},  \\
  \\
-\Delta \phi = \varepsilon \rho(x)u^{2}.
\end{array}
\right.
\end{equation*}
However we are not sure whether this problem has a physical meaning.

\subsection{Concentration on Keplerian orbits}
An interesting question related to \cite{GayDelandeBommier,NAU}, concerns the existence of solutions concentrating on Kepler orbits, assuming radial potentials. For the reasons described in the Introduction, this might be a typical situation where the correspondence principle can be checked using solutions localized on classical planar orbits. We wonder if this result could be obtained for the $3D$ nonlinear Schr\"odinger and Schr\"odinger-Poisson equations with radial potentials.

\subsection{Concentration driven by $\rho$}
If $V \equiv K \equiv 1$, it is natural to ask whether there still exist solutions with a concentration
behaviour. In this case, we expect the location of the concentration points to be governed by the weight $\rho$.
The asymptotic analysis seems more delicate since it requires higher order estimates.

\vspace{0.5cm}
\textbf{Acknowledgements} The authors would like to thank Professor Antonio Ambrosetti for taking their attention
to these questions. C.M. would like to thank the members of the department of Mathematics of Universit{\'e} Libre de
Bruxelles for the kind hospitality and friendship. C.M. was partially supported by FIRB Analysis and Beyond and PRIN
2008 Variational Methods and Nonlinear Differential Equations.


\end{document}